\documentclass[11pt]{article}
\usepackage[a4paper,width=160mm,height=250mm]{geometry}
\usepackage{amssymb}
\usepackage{amsmath}
\usepackage{latexsym}
\usepackage{theorem}
\usepackage{graphicx} 
\usepackage[applemac]{inputenc}
\usepackage{eucal}
\usepackage[sans]{dsfont}
\usepackage{enumerate}
\usepackage[matrix,arrow,curve,cmtip]{xy}
\usepackage{titlesec}

\titleformat{\section}[hang]{\bf\Large}{\thesection.}{1ex}{}
\titleformat{\subsection}[hang]{\bfseries\normalsize}{\thesubsection}{1ex}{}

\def\distsign{\begin{picture}(0,0)\put(0,0){\circle{4}}\end{picture}}
\def\dist{\mbox{$\xymatrix@1@C=5mm{\ar@{->}[r]|{\distsign}&}$}}
\newcommand\adj[2]{\xymatrix@C=8ex{\ar@{}[r]|{\perp}\ar@/^2ex/[r]^{#1} & \ar@/^2ex/[l]^{#2}}}
\newcommand\jda[2]{\xymatrix@C=8ex{\ar@{}[r]|{\top}\ar@/^2ex/[r]^{#1} & \ar@/^2ex/[l]^{#2}}}
\newcommand\arr[1]{\xymatrix{\ar@{->}[r]^{#1}&}}

\newtheorem{theorem}{Theorem}[subsection]
\newtheorem{definition}[theorem]{Definition} 
\newtheorem{proposition}[theorem]{Proposition}
\newtheorem{corollary}[theorem]{Corollary}
\newtheorem{lemma}[theorem]{Lemma}
{\theorembodyfont{\upshape}\newtheorem{example}[theorem]{Example}}
{\theorembodyfont{\upshape}\newtheorem{remark}[theorem]{Remark}}
\newenvironment{proof}[1][Proof.]{\begin{trivlist}\item[\hskip\labelsep{\it #1}]}{\hfill$\Box$\end{trivlist}}

\newcommand\Sup{\mathsf{Sup}}
\newcommand\Quant{\mathsf{Quant}}
\renewcommand\P{\mathcal{P}}
\newcommand\R{\mathcal{R}}
\renewcommand\S{\mathcal{S}}
\newcommand\T{\mathcal{T}}
\renewcommand\:{\colon}
\newcommand\sw{\!\swarrow\!}
\newcommand\se{\!\searrow\!}
\newcommand\lax{_{\mathsf{lax}}}
\newcommand\sq{\!\rightsquigarrow\!}
\renewcommand\S{\mathcal{S}}
\renewcommand\o{^{\sf o}}
\newcommand\Cat{{\sf Cat}}
\newcommand\Ord{{\sf Ord}}
\newcommand\impl{\Rightarrow}
\renewcommand\2{{\bf 2}}
\newcommand\op{^{\sf op}}
\newcommand\dom{{\sf dom}}
\newcommand\cod{{\sf cod}}
\newcommand\Dist{{\sf Dist}}
\newcommand\id{{\sf id}}
\newcommand\cc{_{\sf cc}}
\newcommand\s{_{\sf s}}
\newcommand\Q{{\cal Q}}
\newcommand\V{{\cal V}}
\renewcommand\H{{\cal H}}
\newcommand\bbA{\mathbb{A}}
\newcommand\bbB{\mathbb{B}}
\newcommand\bbC{\mathbb{C}}
\newcommand\bbD{\mathbb{D}}
\newcommand\bbE{\mathbb{E}}
\newcommand\bbT{\mathbb{T}}
\newcommand\tensor{\otimes}
\def\eqref#1{(\ref{#1})}
\newcommand\D{\mathcal{D}}
\renewcommand\1{\mathds{1}}
\newcommand\+{^{\dagger}}
\newcommand\bbS{\mathbb{S}}
\newcommand\Cocont{\mathsf{Cocont}}
\newcommand\Sym{{\sf Sym}}
\newcommand\Cont{{\sf Cont}}
\newcommand\bbX{\mathbb{X}}
\newcommand\bbY{\mathbb{Y}}
\newcommand\Met{\mathsf{Met}}
\newcommand\Gen{\mathsf{G}}
\newcommand\Par{\mathsf{P}}
\newcommand\Ult{\mathsf{U}}
\def\cl#1{\overline{#1}} % closure
\def\scl#1{\widehat{#1}} % symmetric closure
\newcommand\hookto{\hookrightarrow} % full embedding
\newcommand\Clos{\mathsf{Clos}}
\newcommand\gd{_{\mathsf{gd}}}
\newcommand\nz{_{\mathsf{nz}}}
\newcommand\z{_{\mathsf{z}}}

\begin{document}

\title{Topology from enrichment: the curious case of partial metrics} 
\author{Dirk Hofmann\thanks{Departamento de Matemática, Universidade de Aveiro, Portugal, {\tt dirk@ua.pt}; Research partially supported by Centro de Investiga\c{c}\~ao e Desenvolvimento em Matem\'atica e Aplica\c{c}\~oes (CIDMA) da Universidade de Aveiro/FCT -- UID/MAT/04106/2013.} \ and Isar Stubbe\footnote{Lab.\ Math.\ Pures Appl., Université du Littoral, France, {\tt isar.stubbe@univ-littoral.fr}}}

\date{}

\maketitle 

%% for easy of reference (to be removed in final version):
%\vspace{-7cm}\begin{flushright}
%{\footnotesize\sl File: \jobname.tex}
%\end{flushright}
%\vspace{6.5cm}
%%

\begin{quote}{\small
{\bf Abstract.} For any small quantaloid $\Q$, there is a new quantaloid $\D(\Q)$ of diagonals in $\Q$. If $\Q$ is divisible then so is $\D(\Q)$ (and vice versa), and then it is particularly interesting to compare categories enriched in $\Q$ with categories enriched in $\D(\Q)$. Taking Lawvere's quantale of extended positive real numbers as base quantale, $\Q$-categories are generalised metric spaces, and $\D(\Q)$-categories are generalised partial metric spaces, i.e.\ metric spaces in which self-distance need not be zero and with a suitably modified triangular inequality. We show how every small quantaloid-enriched category has a canonical closure operator on its set of objects: this makes for a functor from quantaloid-enriched categories to closure spaces. Under mild necessary-and-sufficient conditions on the base quantaloid, this functor lands in the category of topological spaces; and an involutive quantaloid is Cauchy-bilateral (a property discovered earlier in the context of distributive laws) if and only if the closure on any enriched category is identical to the closure on its symmetrisation. As this now applies to metric spaces and partial metric spaces alike, we demonstrate how these general categorical constructions produce the ``correct'' definitions of convergence and Cauchyness of sequences in generalised partial metric spaces. Finally we describe the Cauchy-completion, the Hausdorff contruction and exponentiability of a partial metric space, again by application of general quantaloid-enriched category theory.
}\end{quote}
%
% 06A15  	Galois correspondences, closure operators
% 06F07  	Quantales
% 18D20  	Enriched categories (over closed or monoidal categories)
% 54E35  	Metric spaces, metrizability

\section{Introduction}\label{intro}

Following Fr\'echet \cite{fre06}, a metric space $(X,d)$ is a set $X$ together with a real-valued function $d$ on $X\times X$ such that the following axioms hold:

[M0] $d(x,y)\geq 0$,

[M1] $d(x,y)+d(y,z)\geq d(x,z)$,

[M2] $d(x,x)=0$,

[M3] if $d(x,y)=0=d(y,x)$ then $x=y$,

[M4] $d(x,y)=d(y,x)$,

[M5] $d(x,y)\neq+\infty$.\\
The categorical content of this definition, as first observed by Lawvere \cite{law72}, is that the extended real interval $[0,\infty]$ underlies a quantale $([0,\infty],\bigwedge,+,0)$, so that a ``generalised metric space'' (i.e.\ a structure as above, minus the axioms M3-M4-M5) is exactly a category enriched in that quantale.

More recently, see e.g.\  \cite{mat94}, the notion of a partial metric space $(X,p)$ has been proposed to mean a set $X$ together with a real-valued function $p$ on $X\times X$ satisfying the following axioms:

[P0] $p(x,y)\geq 0$,

[P1] $p(x,y)+p(y,z)-p(y,y)\geq p(x,z)$,

[P2] $p(x,y)\geq p(x,x)$,

[P3] if $p(x,y)=p(x,x)=p(y,y)=p(y,x)$ then $x=y$,

[P4] $p(x,y)=p(y,x)$,

[P5] $p(x,y)\neq+\infty$.\\
The categorical content of {\em this} definition was discovered in two steps: first, H\"ohle and Kubiak \cite{hohkub11} showed that there is a particular quantaloid of positive real numbers, such that categories enriched in that quantaloid correspond to (``generalised'') partial metric spaces; and second, we realised in \cite{stu13} that H\"ohle and Kubiak's quantaloid of real numbers is actually a universal construction on Lawvere's quantale of real numbers: namely, the quantaloid $\mathcal{D}[0,\infty]$ of diagonals in $[0,\infty]$.

It was shown in \cite{hoftho10} that to any category enriched in a symmetric quantale one can associate a closure operator on its collection of objects. For a metric space $(X,d)$, viewed as an $[0,\infty]$-enriched category, that ``categorical closure'' on $X$ coincides precisely with the metric (topological) closure defined by $d$. And Lawvere \cite{law72} famously reformulated the Cauchy completeness of a metric space in terms of adjoint distributors. It is however not that complicated to extend the construction of the ``categorical closure'' to general quantaloid-enriched categories, thus making it applicable to partial metric spaces viewed as $\D[0,\infty]$-enriched categories. And then it is only natural to see if and how Lawvere's arguments for metric spaces go through in the case of partical metrics. This is what we set out to do in this paper---whence its title.

Here is a brief overview of the contents of this paper. Section \ref{A} contains a compact presentation of well-known quantaloid-enriched category theory \cite{stu05} that we shall need further on. In Section \ref{B} we first explain the construction of the quantaloid $\D(\Q)$ of {\em diagonals} in a given quantaloid, to then recall (and somewhat improve) the closely related notion of {\em divisible quantaloid} as it first appeared in \cite{stu13}. For the sake of exposition, we shall say that a $\D(\Q)$-enriched category is a {\em partial $\Q$-enriched category}, and at the end of Section \ref{B} we explain how (generalised) {\em partial} metric spaces are, indeed, precisely the {\em partial} $[0,\infty]$-enriched categories. We start Section \ref{C} by explaining how every quantaloid-enriched category determines a {\em categorical closure} on its set of objects (generalising results from \cite{hoftho10}); we furthermore characterise those quantaloids for which the closure on any enriched category $\bbC$ is topological, and those involutive quantaloids for which the closure on any enriched category $\bbC$ is always identical to the closure on the symmetrised enriched category $\bbC\s$. Viewing partial metric spaces as enriched categories, we identify in Section \ref{topparmet} the categorical topology induced by a (finitely typed) partial metric---and prove that it is always metrisable by means of a symmetric metric. We spell out what it means for a sequence to converge, resp.\ to be Cauchy, in such a partial metric space $(X,p)$, and then show that all such Cauchy sequences converge in $(X,p)$ if and only if all Cauchy distributors on $(X,p)$ qua enriched category are representable. We end with some examples concerning Hausdorff distance in, and exponentiability of, partial metric spaces.

Of course, the study of partial metrics is not new. For example, in the survey paper \cite{bukatinetal09} partial metrics are studied by analogy with metrics, and the reader will find there e.g.\ the definition of Cauchy sequence in a (symmetric) partial metric space (where $p(x,y)\neq\infty$). Let us also mention that \cite{puzha12} already adopts an enriched category point of view, and shows how those Cauchy sequences correspond with Cauchy distributors. However, none of the previously published papers have our purely categorical setup: we contruct a {\it topology} for any quantaloid-enrichment, so that -- when applied to the quantaloid of diagonals in $[0,\infty]$ -- the generic topological notions of convergence and Cauchyness of sequences reproduce those that were considered in a rather {\it ad hoc} manner before. So whereas our paper does not present many new results in partial metric spaces {\it per se}, it does propose a whole new categorical method to study partial metrics. That this method is benificial, can be seen in our treatment of a hitherto undiscovered subtlety involving the points with self-distance $\infty$, and in our results on Hausdorff distance and exponentiability!

\section{Preliminaries, exemplified by metric spaces and ordered sets}\label{A}

A large part of the general theory of quantales and quantaloids is an instance of {\em $\V$-enriched category theory} \cite{kel82}, taking the base category $\V$ to be the category $\Sup$ of complete lattices and supremum-preserving functions. Indeed, a {\em quantaloid $\Q$} is a $\Sup$-enriched category (and quantales are exactly quantaloids with a single object, i.e.\ monoids in $\Sup$), a {\em homomorphism $H\:\Q\to\R$} between quantaloids is a $\Sup$-enriched functor, and so forth. 

On the other hand, quantaloids are also (very particular) {\em bicategories} \cite{ben67}, so the general notions from bicategory theory apply as well. This point of view is important when defining lax morphisms between quantaloids, or adjunctions in a quantaloid, or quantaloid-enriched categories, because these concepts are not ``naturally'' catered for by $\Sup$-enriched category theory alone.

In his seminal paper \cite{law72}, Lawvere shows how both metric spaces and ordered sets are a guiding example of enriched categories---quantale-enriched, that is. In this section we shall reproduce some of his insightful examples, but we do explain the (slightly) more general case of quantaloid-enrichment, for in the next section it will be crucial to have this ready for the case of partial metric spaces (as will become clear there). For clarity's sake, and to fix our notations, we shall spell out some of these abstract categorical definitions in more elementary terms.

\subsection{Quantaloids, quantales}

A {\em quantaloid $\Q$} is a category in which, for any two fixed objects $A$ and $B$, the set $\Q(A,B)$ of morphisms from $A$ to $B$ is ordered and admits all suprema, in such a way that composition distributes on both sides over arbitrary suprema: whenever $f\:A\to B$, $(g_i\:B\to C)_{i\in I}$ and $h\:C\to D$, then $h\circ(\bigvee_ig_i)\circ g=\bigvee_i(h\circ g_i\circ f)$. We write $1_A\:A\to A$ for the identity morphism on an object $A$.

A crucial property of quantaloids is their so-called {\em closedness}. Precisely, for any morphism $f\:A\to B$ in a quantaloid $\Q$ and any object $X$ of $\Q$, both
$$-\circ f\:\Q(B,X)\to\Q(A,X)\mbox{ and }f\circ -\:\Q(X,A)\to\Q(X,B)$$
({\em pre-} and {\em post-composition with $f$}) are supremum-preserving functions between complete lattices. Therefore these maps have right adjoints, called {\em lifting} and {\em extension through $f$}, and we shall write these as:
$$-\sw f\:\Q(A,X)\to\Q(B,X)\mbox{ and }f\se-\:\Q(X,B)\to\Q(X,A).$$

A {\em quantale $Q$} is, by definition, a one-object quantaloid. Equivalently, a quantale $Q=(Q,\bigvee,\circ,1)$ is a sup-lattice $(Q,\bigvee)$ equipped with a monoidal structure $(\circ,1)$ in such a way that multiplication distributes on both sides over suprema. Liftings and extensions in a quantale are often called {\em (left/right) residuations}, especially in the context of multi-valued logics.

The above says in particular that a quantaloid is a {\em locally complete and cocomplete closed bicategory} (and a quantale is a complete and cocomplete closed monoidal category). Importantly, we can therefore use all bicategorical notions in any quantaloid: adjoint pairs, monads, 2-dimensional universal properties, etc.
\begin{example}\label{a12}
Any locale (= complete Heyting algebra) $H$ is a quantale $H=(H,\wedge,\top)$. In fact, cHa's are precisely those quantales which are integral (meaning that $1=\top$) and idempotent (meaning that $f^2=f$ for every $f\in Q$); they are of course also commutative. In particular shall we write $\2=(\2,\wedge,1)$ for the 2-element Boolean algebra $\{0<1\}$ viewed as quantale. 
\end{example}
\begin{example}\label{b1}
Writing $[0,\infty]\op$ for the set of positive real numbers extended with $+\infty$, with the {\it opposite} of the natural order, it is easy to check that $R=([0,\infty]\op,+,0)$ is a (commutative and integral, but not idempotent) quantale; throughout this article we shall refer to it as {\em Lawvere's quantale of positive real numbers}, to honor its first appearence in \cite{law72}.
\end{example}
In the remainder of this section, these quantales will be our main examples. It is however necessary to develop the general quantaloidal case for reasons that will become clear in the next section. To give but one example of a non-commutative, non-integral, non-idempotent quantale, consider the set $\Sup(L,L)$ of supremum-preserving functions on a complete lattice $L$. In fact, the category $\Sup$ of complete lattices and supremum-preserving morphisms itself is the example {\it par excellence} of a (large) quantaloid. Many more examples can be found in the references.

\subsection{Quantaloid-enriched categories, functors, distributors}\label{abc}

In all that follows, we fix a small quantaloid $\Q$; we shall write $\Q_0$ for its set of objects and $\Q_1$ for its set of morphisms. 

A {\em $\Q$-enriched category $\mathbb{C}$} (or {\em $\Q$-category $\bbC$} for short) consists of 
\begin{itemize}\setlength{\itemindent}{3em}
\item[{\tt (obj)}] a set $\mathbb{C}_0$ of ``objects'',
\item[{\tt (typ)}] a unary ``type'' predicate $t\:\bbC_0\to\Q_0\:x\mapsto tx$,
\item[{\tt (hom)}] a binary ``hom'' precidate $\bbC\:\bbC_0\times\bbC_0\to\Q_1\:(x,y)\mapsto\bbC(x,y)$,
\end{itemize}
such that the following conditions hold:
\begin{itemize}\setlength{\itemindent}{3em}
\item[{\tt [C0]}] $\mathbb{C}(x,y)\colon ty\to tx$,
\item[{\tt [C1]}] $\mathbb{C}(x,y)\circ\mathbb{C}(y,z)\leq\mathbb{C}(x,z)$,
\item[{\tt [C2]}] $1_{tx}\leq\mathbb{C}(x,x)$.
\end{itemize}
Note how, when applied to a quantale $Q$ (viewed as a one-object quantaloid $\Q$), the above definition symplifies: the ``type'' predicate becomes obsolete and condition {\tt [C0]} trivialises. This is the case in our main examples (for now):
\begin{example}\label{c}
Let $\2=(\2,\wedge,\top)$ be the 2-element Boolean algebra. A $\2$-category $\bbA$ is exactly an ordered\footnote{An {\em order} is a transitive and reflexive relation; if we want it to be anti-symmetric, then we shall explicitly mention so.} set: for we may interpret that $\bbA(x,y)\in\{0,1\}$ is 1 if and only if $x\leq y$. 
\end{example}
\begin{example}\label{d}
Considering the Lawvere quantale $R=([0,\infty]\op,+,0])$, an $R$-category $\bbX$ consists of a set $X=\bbX_0$ together with a function $d=\bbX(-,-)\:X\times X\to[0,\infty]$ such that $d(x,y)+d(y,z)\geq d(x,z)$ and $0=d(x,x)$; all other data and conditions are trivially satisfied. Such an $(X,d)$ is a {\em generalised metric space} \cite{law72}; adding symmetry ($d(x,y)=d(y,x)$), separatedness (if $d(x,y)=0=d(y,x)$ then $x=y$) and finiteness ($d(x,y)<\infty$) makes it a metric space in the sense of Fréchet \cite{fre06}\footnote{In the introduction to the 2002 reprint of his \cite{law72}, Lawvere explains that ``the evidence is compelling that the usually-given more restrictive definition [of `metric space'] was too hastily fixed'', for a ``metric space can always be symmetrized [...], but it is often better to delay that until the last stage of a calculation. [...] Likewise, metric spaces need not have all distances finite, but one can (when appropriate) restrict consideration to those points which have finite distance to a given part. The coproducts [of metric spaces] have infinite distance between points in different summands''.}.
\end{example}
A {\em $\Q$-functor} $F\:\bbC\to\bbD$ between two $\Q$-categories is 
\begin{itemize}
\setlength{\itemindent}{3em}
\item[{\tt (map)}] an ``object map'' $F\:\bbC_0\to\bbD_0\:x\mapsto Fx$
\end{itemize}
satisfying, for all $x,x'\in\bbC_0$,
\begin{itemize}
\setlength{\itemindent}{3em}
\item[{\tt (F0)}] $t(Fx)=tx$,
\item[{\tt (F1)}] $\bbC(x',x)\leq\bbD(Fx',Fx)$.
\end{itemize}
Such $\Q$-functors $F\:\bbA\to\bbB$ and $G\:\bbB\to\bbC$ can be composed in the obvious way to produce a new functor $G\circ F\:\bbA\to\bbC$, and the identity object map provides for the identity functor $1_{\bbA}\:\bbA\to\bbA$. Thus $\Q$-categories and $\Q$-functors are the objects and morphisms of a (large) category $\Cat(\Q)$. 
\begin{example}\label{e}
There is no difficulty in proving that $\Cat(\2)$ is exactly $\Ord$, the category of ordered sets and order-preserving functions.
\end{example}
\begin{example}\label{f}
Upon identifying two $R$-categories $\bbX$ and $\bbY$ with two generalised metric spaces $(X,d_X)$ and $(Y,d_Y)$, it is straightforward to verify that an $R$-functor $F\:\bbX\to\bbY$ can be identified with a 1-Lipschitz function $f\:X\to Y$, i.e.\ $d_X(x',x)\geq d_Y(fx',fx)$. We shall write $\Gen\Met$ for the category $\Cat(R)$.
\end{example}
To make $\Q$-enriched category theory really interesting, we need to introduce a second kind of morphism between $\Q$-categories: a {\em $\Q$-distributor}\footnote{We use the terminology of \cite{ben67}; the same concept has been named `module' or `bimodule' (particularly by the Australian category theorists) and `profunctor' (particularly in the context of proarrow equipments).} $\Phi\:\bbC\dist\bbD$ between two $\Q$-categories is 
\begin{itemize}\label{g}
\setlength{\itemindent}{3em}
\item[{\tt (matr)}] a ``matrix'' $\Phi\:\bbD_0\times\bbC_0\to\Q_1\:(y,x)\mapsto\Phi(y,x)$
\end{itemize}
satisfying, for all $x,x'\in\bbC_0$ and $y,y'\in\bbD_0$,
\begin{itemize}
\setlength{\itemindent}{3em}
\item[{\tt (D0)}] $\Phi(y,x)\:tx\to ty$,
\item[{\tt (D1)}] $\bbD(y',y)\circ\Phi(y,x)\leq\Phi(y',x)$,
\item[{\tt (D2)}] $\Phi(y,x)\circ\bbC(x,x')\leq\Phi(y,x')$.
\end{itemize}
For two consecutive distributors $\Phi\:\bbC\dist\bbD$ and $\Psi\:\bbD\dist\bbE$, the composite distributor is written as $\Psi\tensor\Phi\:\bbC\dist\bbE$ and computed as, for $x\in\bbC$ and $z\in\bbE$,
$$(\Psi\tensor\Phi)(z,x)=\bigvee_{y\in\bbD_0}\Psi(z,y)\circ\Phi(y,x).$$
The identity distributor $\id_{\bbC}\:\bbC\dist\bbC$ has elements, for $x,x'\in\bbC_0$,
$$\id_{\bbC}(x',x)=\bbC(x',x).$$
Two parallel distributors $\Phi,\Phi'\:\bbC\dist\bbD$ are ordered `elementwise': 
$$\Phi\leq\Phi'\stackrel{\rm def}{\iff}\Phi(y,x)\leq\Phi'(y,x)\mbox{ for all }(x,y)\in\bbC_0\times\bbD_0,$$
and therefore the supremum of a family of parallel distributors, say $\Phi_i\:\bbC\dist\bbD$, has elements
$$(\bigvee_i\Phi_i)(y,x)=\bigvee_i\Phi_i(y,x).$$
In this manner, distributors are the morphisms of a (large) quantaloid $\Dist(\Q)$.
\begin{example}\label{h}
A distributor between ordered sets $(X,\leq)$ and $(Y,\leq)$ (viewed as $\2$-enriched categories) is precisely a relation $R\subseteq Y\times X$ which is upclosed in $X$ and downclosed in $Y$: if $y'\leq y R x\leq x'$ then $y'Rx$.
\end{example}
The importance of $\Dist(\Q)$ being a {\em quantaloid} -- instead of a mere {\em category} -- cannot be overestimated: for it implies that $\Dist(\Q)$ is closed, that we can speak of adjoint pairs of distributors, that we can perform 2-categorical constructions involving $\Q$-categories and distributors, and so on. For instance, it is not difficult to verify that we can compute liftings and extensions in $\Dist(\Q)$ by the following formulas:
\begin{equation}\label{i}
\begin{tabular}{cc}
$\begin{array}{c}
\xymatrix@=5ex{
\bbC\ar[dr]|{\distsign}_{\Phi}\ar@{.>}[rr]|{\distsign}^{\Psi\searrow\Phi} & &\bbD\ar[dl]|{\distsign}^{\Psi} \\
 & \bbE}
\end{array}$
&
$\displaystyle(\Psi\searrow\Phi)(y,x)=\bigwedge_{z\in\bbE_0}\Psi(z,y)\searrow\Phi(z,x)$
\end{tabular}
\end{equation}
\begin{equation}\label{j}
\begin{tabular}{cc}
$\begin{array}{c}
\xymatrix@=5ex{
\bbC\ar@{.>}[rr]|{\distsign}^{\Psi\swarrow\Phi} & &\bbD \\
 & \bbE\ar[ul]|{\distsign}_{\Phi}\ar[ur]|{\distsign}^{\Psi}}
\end{array}$
&
$\displaystyle(\Psi\swarrow\Phi)(y,x)=\bigwedge_{z\in\bbE_0}\Psi(y,z)\swarrow\Phi(x,z)$
\end{tabular}
\end{equation}
In contrast, there is {\it a priori} no extra structure in $\Cat(\Q)$---but luckily $\Cat(\Q)$ embeds naturally in $\Dist(\Q)$, and therefore inherits some of the latter's structure. 

Indeed, every functor $F\:\bbA\to\bbB$ determines an adjoint pair of distributors
$$\xymatrix@C=8ex{\bbA\ar@{}[r]|{\perp}\ar@/^2ex/[r]|{\distsign}^{F_*} & \bbB\ar@/^2ex/[l]|{\distsign}^{F^*}}$$
defined by $F_*(b,a)=\bbB(b,Fa)$ and $F^*(a,b)=\bbB(Fa,b)$. We shall say that the left adjoint $F_*$ is the {\em graph} of the functor $F$, whereas the right adjoint $F^*$ is its {\em cograph}. Taking graphs and cographs extends to a pair of functors, one covariant and the other contravariant:
\begin{equation}\label{a1}
\Cat(\Q)\to\Dist(\Q)\:\Big(F\:\bbA\to\bbB\Big)\mapsto\Big(F_*\:\bbA\dist\bbB\Big),
\end{equation}
\begin{equation}\label{a1.1}
\Cat(\Q)\op\to\Dist(\Q)\:\Big(F\:\bbA\to\bbB\Big)\mapsto\Big(F^*\:\bbB\dist\bbA\Big).
\end{equation}
With this, we make $\Cat(\Q)$ a {\em locally ordered category} by defining, for any parallel pair of functors $F,G\:\bbA\to\bbB$,
$$F\leq G\stackrel{\rm def}{\iff}F_*\leq G_*\iff F^*\geq G^*.$$
Whenever $F\leq G$ and $G\leq F$, we write $F\cong G$ and say that these functors are isomorphic.

With all this, we can now naturally speak of adjoint $\Q$-functors, fully faithful $\Q$-functors, equivalent $\Q$-categories, (co)monads on $\Q$-categories, etc. 

\subsection{Presheaves and completions}\label{X}

If $X$ is an object of $\Q$, then we write $\1_X$ for the $\Q$-category defined by $(\1_X)_0=\{*\}$, $t*=X$ and $\1_X(*,*)=1_X$. Similarly, if $f\:X\to Y$ is a morphism in $\Q$, then we write $(f)\:\1_X\dist\1_Y$ for the distributor defined by $(f)(*,*)=f$. In doing so we get an injective homomorphism
\begin{equation}\label{a2}
i\:\Q\to\Dist(\Q)\:\Big(f\:X\to Y\Big)\mapsto\Big((f)\:\1_X\dist\1_Y\Big)
\end{equation}
which allows us to (tacitly) identify $\Q$ with its image in $\Dist(\Q)$.

A {\em contravariant $\Q$-presheaf $\phi$} of type $X\in\Q_0$ on a $\Q$-category $\bbC$ is, by definition, a distributor $\phi\:\1_X\dist\bbC$. For two such presheaves $\phi\:\1_X\dist\bbC$ and $\psi\:\1_Y\dist\bbC$, the lifting $(\psi\searrow\phi)\:\1_X\dist\1_Y$ in the quantaloid $\Dist(\Q)$ is a distributor with a single element, which can therefore be identified with an arrow from $X$ to $Y$ in $\Q$:
$$\begin{array}{c}\xymatrix@C=3ex{ & {\1_Y}\ar[dr]|{\distsign}^{\psi} \\ {\1_X}\ar[rr]|{\distsign}_{\phi}\ar@{.>}[ur]|{\distsign}^{\psi\searrow\phi} & &\bbC}\end{array}$$
We can thus define the {\em $\Q$-category $\P\bbC$ of contravariant presheaves} on $\bbC$ to have as objects the contravariant presheaves on $\bbC$ (of all possible types); the type of a presheaf $\phi\:\1_X\dist\bbC$ is $X$; and the hom $\P\bbC(\psi,\phi)$ is the (single element of the) lifting $\psi\searrow\phi$.

For $\Phi\:\bbC\dist\bbD$ it is easy to see that $\P\bbC\to\P\bbD\:\psi\mapsto\Phi\tensor\psi$ is a $\Q$-functor. This action easily extends to form a 2-functor $\P\:\Dist(\Q)\to\Cat(\Q)$, and by composition with the inclusion 2-functor $\Cat(\Q)\to\Dist(\Q)$ of \eqref{a1} we find a functor $\P\:\Cat(\Q)\to\Cat(\Q)$. The latter turns out to be a KZ-doctrine (i.e.\ a 2-monad such that ``algebras are adjoint to units''); its category of algebras is denoted $\Cocont(\Q)$: its objects are so-called cocomplete $\Q$-categories, and its morphisms are the cocontinuous $\Q$-functors. The unit of the KZ-doctrine consists of the so-called (fully faithful) {\em Yoneda embeddings}
$$Y_{\bbC}\:\bbC\to\P\bbC\:x\mapsto\bbC(-,x).$$
The presheaves in the image of $Y_{\bbC}$ are said to be {\em representable} (by objects of $\bbC$); the Yoneda embedding $Y_{\bbC}$ exhibits $\P\bbC$ to be the {\em free cocompletion} of $\bbC$. And the {\em Yoneda Lemma} says that, for any $\phi\in\P\bbC$ and any $x\in\bbC$, we have $\P\bbC(Y_{\bbC}x,\phi)=\phi(x)$.

Dually, a covariant $\Q$-presheaf $\kappa$ of type $X\in\Q_0$ on a $\Q$-category $\bbC$ is a distributor like $\kappa\:\bbC\dist\1_X$; they are the objects of a $\Q$-category $\P\+\bbC$, in which $\P\+\bbC(\lambda,\kappa)=\lambda\swarrow\kappa$. The obvious 2-functor $\P\+\:\Dist(\Q)\to\Cat(\Q)\op$ composes with the inclusion 2-functor in \eqref{a1.1} to form a co-KZ-doctrine $\P\+\:\Cat(\Q)\to\Cat(\Q)$, whose category of algebras $\Cont(\Q)$ consists of complete $\Q$-categories and continuous $\Q$-functors. The (fully faithful) Yoneda embeddings (also: free completions)
$$Y\+_{\bbC}\:\bbC\to\P\+\bbC\:x\mapsto\bbC(x,-),$$
form the unit of the co-KZ-doctrine.
\begin{example}\label{k}
For an ordered set $(X,\leq)$, viewed as a $\2$-category, a contravariant presheaf on $(X,\leq)$ is exactly a downclosed subset of $X$. For two contravariant presheaves $\phi,\psi$ on $(X,\leq)$, one straightforwardly computes that $\psi\searrow\phi\in\{0,1\}$ is equal to $1$ if and only if $\psi\subseteq\phi$ (as subsets of $X$). That is to say, the free cocompletion $\P(X,\leq)$ -- in the sense of $\2$-category theory -- is precisely the free sup-lattice on $(X,\leq)$: the set of downclosed subsets ordered by inclusion; and the Yoneda embedding $Y_{(X,\leq)}\:(X,\leq)\to\P(X,\leq)$ sends an element $x\in X$ to the principal downclosed set $\downarrow\! x$. Hence, upon identifying $\Ord$ with $\Cat(\2)$, the KZ-doctrine $\P\:\Ord\to\Ord$ is the free sup-lattice monad. In an entirely dual fashion, $\P\+(X,\leq)$ is the free inf-lattice: its elements are the upclosed subsets of $X$, ordered by containment (the {\it opposite} of inclusion); and the co-KZ-doctrine $\P\+\:\Ord\to\Ord$ is the free inf-lattice monad.
\end{example}

A {\em Cauchy presheaf} on a $\Q$-category $\bbC$ is a (contravariant) presheaf $\phi\:\1_X\dist\bbC$ which -- as morphism in $\Dist(\Q)$ -- has a right adjoint, which we shall then write as $\phi^*\:\bbC\dist\1_X$. The $\Q$-category $\bbC\cc$ is, by definition, the full subcategory of $\P\bbC$ whose objects are the Cauchy presheaves. Furthermore, the Yoneda embedding $Y_{\bbC}\:\bbC\to\P\bbC$ co-restricts to a functor $I_{\bbC}\:\bbC\to\bbC\cc$, which is now called the {\em Cauchy completion of $\bbC$}. Those Cauchy completions form the unit of a KZ-doctrine $(-)\cc\:\Cat(\Q)\to\Cat(\Q)$; the category of algebras $\Cat(\Q)\cc$ contains the Cauchy complete $\Q$-categories and (all) $\Q$-functors between them. In fact, a $\Q$-category $\bbC$ is Cauchy complete if and only if, for every left adjoint distributor $\Phi\:\bbX\dist\bbC$, there exists a (necessarily essentially unique) functor $F\:\bbX\to\bbC$ such that $F_*=\Phi$, if and only if for every Cauchy presheaf $\phi\:\1_X\dist\bbC$, there exists a (necessarily essentially unique) object $c\in\bbC_0$ such that $\bbC(-,c)=\phi$. 
\begin{example}\label{l}
The designations ``Cauchy completion'' and ``Cauchy complete'' are motivated by the interpretation of these concepts in generalised metric spaces \cite{law72}, as follows. Every Cauchy sequence $(x_n)_n$ in a metric space $(X,d)$ -- suitably viewed as an $R$-category $\bbX$, cf.\ Example \ref{d} -- defines an adjoint pair $\phi\dashv\phi^*$ of $R$-distributors $\phi\:\1\dist\bbX$ and $\phi^*\:\bbX\dist\1$ by putting
$$\phi(y)=\lim_{n\to\infty} d(y,x_n)\mbox{ \quad and \quad }\phi^*(y)=\lim_{n\to\infty} d(x_n,y)$$
for all $y\in X$. Moreover, $(x_n)_n$ converges to $x\in X$ precisely when $\phi=d(-,x)$. Conversely, a left adjoint $R$-distributor $\phi:\1\dist X$, with right adjoint $\phi^*:X\dist\1$, satisfies
$$\bigwedge_{x\in X}\phi^*(x)+\phi(x)=0,$$
so that a Cauchy sequence $(x_n)_n$ can be built by choosing $x_n$ with $\phi(x_n)+\phi^*(x_n)\le\frac{1}{n}$. After identifying equivalent Cauchy sequences, these two processes turn out to be inverse to each other; and therefore {\em a generalised metric space is Cauchy complete in the traditional sense if and only if it is Cauchy complete in the sense of enriched category theory}.
\end{example}

Such (Cauchy) (co)complete $\Q$-categories can be studied and characterised in many different ways, and have a wealth of applications; we refer to \cite{bonbrerut98,gar14,wal81} for examples.

\subsection{Involution and symmetry}

In this subsection we shall suppose that $\Q$ is a quantaloid equipped with an {\em involution}: a function $\Q_1\to\Q_1\:f\mapsto f\o$ on the morphisms of $\Q$ such that $f\leq g$ implies $f\o\leq g\o$, $(g\circ f)\o=f\o\circ g\o$, and $f^{\sf oo}=f$. It is easy to check that this automatically extends to a (necessarily invertible) homomorphism $(-)\o\:\Q\op\to\Q$ which is the identity on objects and satisfies $f^{\sf oo}=f$ for any morphism $f$ in $\Q$. The pair $(\Q,(-)\o)$ is said to form an {\em involutive quantaloid}, but we leave the notation for the involution understood when no confusion can arise. 

A $\Q$-category $\bbA$ is {\em symmetric} if we have, for all $x,y\in\bbA_0$,
\begin{itemize}\setlength{\itemindent}{3em}
\item[{\tt [C4]}] $\bbA(x,y)=\bbA(y,x)\o$.
\end{itemize}
We shall write $\Sym\Cat(\Q)$ for the full sub-2-category of $\Cat(\Q)$ determined by the symmetric $\Q$-categories. The full embedding $\Sym\Cat(\Q)\hookrightarrow\Cat(\Q)$ has a right adjoint functor\footnote{But the right adjoint is not a 2-functor, so this is not a 2-adjunction!},
$$\Sym\Cat(\Q)\xymatrix@=8ex{\ar@{}[r]|{\perp}\ar@<1mm>@/^2mm/[r]^{\mathrm{incl.}} & \ar@<1mm>@/^2mm/[l]^{(-)\s}}\Cat(\Q),$$
which sends a $\Q$-category $\bbC$ to the symmetric $\Q$-category $\bbC\s$ whose objects (and types) are those of $\bbC$, but for any two objects $x,y$ the hom-arrow is 
$$\bbC\s(y,x):=\bbC(y,x)\wedge\bbC(x,y)\o.$$ 
A functor $F\:\bbC\to\bbD$ is sent to $F\s\:\bbC\s\to\bbD\s\:x\mapsto Fx$. The counit of this adjunction has components $S_{\bbC}\:\bbC\s\to\bbC\:x\mapsto x$.
\begin{example}\label{m}
A commutative quantale is the same thing as a quantale for which the identity map is an involution; in particular can we thus consider $\2=\{0,1\}$ and $R=[0,\infty]\op$ to be involutive. For both ordered sets and generalised metric spaces it is straightforward to interpret the symmetry axiom: an order $(X,\leq)$ is symmetric qua $\2$-enriched category if and only if the order-relation $\leq$ is symmetric (and so it is an {\it equivalence relation} on $X$); and a generalised metric space $(X,d)$ is symmetric qua $R$-enriched category if and only if the distance function $d$ is symmetric (and so $(X,d)$ is an {\it écart} in the sense of \cite{bou48}).
\end{example}

Composing right and left adjoint, we find a comonad $(-)\s\:\Cat(\Q)\to\Cat(\Q)$ whose coalgebras are exactly the symmetric $\Q$-categories. In the previous subsection we had the important monad $(-)\cc\:\Cat(\Q)\to\Cat(\Q)$ whose algebras are exactly the Cauchy complete $\Q$-categories. Because both arise from (co)reflexive subcategories, there can be at most one distributive law of the Cauchy monad over the symmetrisation comonad; here is a sufficient condition for its existence:
\begin{proposition}[\cite{heystu11}]\label{n}
If an involutive quantaloid is {\em Cauchy-bilateral}, that is to say, for each family $(f_i\:X\to X_i, g_i\:X_i\to X)_{i\in I}$ of morphisms in $\Q$,
$$\left.\begin{array}{c}
\forall j,k\in I:\ f_k\circ g_j\circ f_j\leq f_k \\[1ex]
\forall j,k\in I:\ g_j\circ f_j\circ g_k\leq g_k \\[1ex]
1_X\leq\displaystyle\bigvee_{i\in I}g_i\circ f_i
\end{array}\right\}\Longrightarrow\
1_X\leq\bigvee_{i\in I}(g_i\wedge f_i\o)\circ(g_i\o\wedge f_i),$$
then there is a distributive law $L$ of $(-)\cc\:\Cat(\Q)\to\Cat(\Q)$ over $(-)\s\:\Cat(\Q)\to\Cat(\Q)$ with components
$$L_{\bbC}\:(\bbC\s)\cc\to(\bbC\cc)\s\:\phi\mapsto(S_{\bbC})_*\tensor\phi.$$
The category of $L$-bialgebras contains precisely those $\Q$-categories which are both symmetric and Cauchy complete, and all $\Q$-functors between these.
\end{proposition}
This means that, for such a Cauchy-bilateral $\Q$, the Cauchy completion of a symmetric $\Q$-category is again symmetric, and that the symmetrisation of a Cauchy complete $\Q$-category is again Cauchy complete; so the category of $L$-bialgebras can be computed, either by Cauchy-completing all symmetric $\Q$-categories, or by symmetrising all Cauchy complete $\Q$-categories. For more details we refer to \cite{heystu11}.
\begin{example}\label{o}
Every locale $H$ is Cauchy-bilateral (for the identity involution); in particular so is $\2$. But in the $\2$-enriched case, every $\2$-category is Cauchy complete!
\end{example}
\begin{example}\label{p}
The Lawvere quantale $R=([0,\infty]\op,+,0)$ is Cauchy-bilateral (again, for the identity involution), so the Cauchy completion of a symmetric generalised metric space is again a symmetric generalised metric metric space. (Perhaps this motivated Fréchet \cite{fre06} to include the symmetry axiom in his definition of `metric space'?)
\end{example}
The above example generalises, as follows:
\begin{example}\label{p.1}
Any linearly ordered, integral, commutative quantale $Q$ is Cauchy-bilateral (for the identity involution). Indeed, in this case the condition to be Cauchy-bilateral reduces to
$$1\leq\bigvee_i g_i\circ f_i\ \Longrightarrow\ 1\leq\bigvee_i(g_i\wedge f_i)^2$$
for any family $(f_i,g_i)_{i\in I}$ of pairs of elements of $Q$. But integrality of $Q$ assures that $g_i\circ f_i\leq g_i\wedge f_i$, so the hypothesis implies that
$$1\leq\bigvee_i g_i\circ f_i\leq\bigvee_i g_i\wedge f_i,$$
and therefore -- taking squares on both ends of this inequality -- also
$$1\leq(\bigvee_i g_i\wedge f_i)^2=\bigvee_{i,j}(g_i\wedge f_i)(g_j\wedge f_j).$$
Now it is linearity of $Q$ which makes $(g_i\wedge f_i)(g_j\wedge f_j)\leq (g_i\wedge f_i)^2\vee(g_j\wedge f_j)^2$, so that from the previous line we easily find the desired result. A {\em left-continuous $t$-norm} \cite{got01} is exactly an integral commutative quantale structure on the (linearly ordered) real unit interval; so here we find in particular all these to be Cauchy-bilateral.
\end{example}

\subsection{Homomorphisms, lax functors, change of base}

A {\em homomorphism $H\:\Q\to\R$} between quantaloids (and in particular quantales) is a functor, mapping $f\:A\to B$ in $\Q$ to $Hf\:HA\to HB$ in $\R$ and preserving composition and identities in the usual manner, which furthermore preserves local suprema: whenever $(f_i\:A\to B)_{i\in I}$ in $\Q$, then $H(\bigvee_if_i)=\bigvee_iHf_i$. (Note that, as a consequence, $H$ preserves local order.)  Homomorphisms $H\:\Q\to\R$ and $K\:\R\to\S$ compose to produce a new homomorphism $K\circ H\:\Q\to\S$, and on each quantaloid $\Q$ there is an identity homomorphism $1_{\Q}$; so (small) quantaloids and homomorphisms themselves are the objects and morphisms of a (large) category $\Quant$.

A {\em lax functor\footnote{What we really define here, is a {\em lax $\Ord$-functor}, i.e.\ a lax functor between quantaloids qua $\Ord$-enriched categories.}} $F\:\Q\to\R$ maps $f\:A\to B$ in $\Q$ to $Ff\:FA\to FB$ in $\R$ in such a way that
\begin{itemize}
\item if $f\leq f'$ in $\Q(A,B)$ then $Ff\leq Ff'$ in $\R(FA,FB)$,
\item if $f\:A\to B$ and $g\:B\to C$ in $\Q$, then $Fg\circ Ff\leq F(g\circ f)$ in $\R$,
\item for any $A$ in $\Q$, $1_{FA}\leq F1_A$ in $\R$.
\end{itemize}
Lax functors compose in the obvious manner, so there is a (large) category $\Quant\lax$ of (small) quantaloids and lax morphisms, containing $\Quant$. If a lax functor preserves all identities, then it is said to be {\em normal}; the composite of such is again a normal lax functor, so these are the morphisms of a category containing $\Quant$ and contained in $\Quant\lax$. (Other variations exist, e.g.\ lax functors which preserve all local infima, or only finite local infima, etc.)

Now suppose that $F\:\Q\to\R$ is a lax functor. If $\bbC$ is any $\Q$-category, then it is straightforward to define an $\R$-category $F\bbC$ with the same object set as $\bbC$ but with homs given by $F\bbC(y,x)=F(\bbC(y,x))$. This construction extends to distributors and functors, producing a 2-functor $\Cat(\Q)\to\Cat(\R)$ and a lax morphism $\Dist(\Q)\to\Dist(\R)$, both referred to as {\em change of base} functors.

For any quantaloid $\Q$ we can define the lax morphism $\Q\to\2$ which (obviously) sends every object of $\Q$ to the single object of $\2$, every arrow bigger or equal to an identity in $\Q$ to the non-zero arrow in $\2$, and all other arrows to zero. The change of base $\Cat(\Q)\to\Cat(\2)=\Ord$ thus associates to any $\Q$-category its {\em underlying ordered set} (and sends $\Q$-functors to monotone functions); precisely, it sends a $\Q$-category $\bbC$ to the order $(\bbC_0,\leq)$ where
$$x\leq y\mbox{ exactly when }tx=ty\mbox{ and }1_{tx}\leq\bbC(x,y).$$
The $\Q$-category $\bbC$ is said to be {\em skeletal} (or {\em separated}) when its underlying order is anti-symmetric. Even when $\bbC$ is not skeletal, $\P\bbC$ (and hence its full subcategory $\bbC\cc$) is.
\begin{example}\label{q}
Tautologically, an order $(X,\leq)$ is skeletal qua $\2$-enriched category if and only if the order-relation $\leq$ is anti-symmetric; in other words, $(X,\leq)$ is a {\it partially ordered set} (but we will avoid that terminology, leaving the adjective `partial' available for something quite different---see Subsection \ref{yz}).
\end{example}
\begin{example}\label{r}
Applied to the Lawvere quantale $R=([0,\infty]\op,+,0)$, the change of base into $\2$ becomes the functor $\Gen\Met\to\Ord$ which sends a generalised metric space $(X,d)$ to the ordered set $(X,\leq)$ in which $x\leq y$ precisely when $d(x,y)=0$. It follows that a generalised metric space $(X,d)$ is skeletal qua $R$-enriched category if and only if the distance function is separating in the (weak\footnote{A stronger sense would be to require that already $d(x,y)=0$ implies $x=y$, as is used in the (not so common) definition of a {\it quasi-metric}; but this stronger separation axiom has no use in the theory of ordered sets (it holds exactly for discrete orders), indicating that it is perhaps not the most useful categorical concept.}) sense that $d(x,y)=0=d(y,x)$ implies $x=y$. A symmetric and skeletal $(X,d)$ is thus the same thing as a metric $d\:X\times X\to [0,\infty]$ (allowing $\infty$). Clearly, if the metric is symmetric and/or separated then its underlying order is so too.
\end{example}
\begin{example}\label{s}
Up to now we have considered the ordered set $[0,\infty]\op$ as a quantale for the addition; and we saw that $([0,\infty]\op,+,0)$-enriched categories are generalised metric spaces. But we can also consider the locale $([0,\infty]\op,\vee,0)$---so it is the same underlying order, but now with binary supremum as binary operation. It is straightforward to check that a $([0,\infty]\op,\vee,0)$-enriched category is exactly a {\it generalised ultrametric space} $(X,d)$, i.e.\ a distance function $d\:X\times X\to[0,\infty]$ satisfying $d(x,x)=0$ and $d(x,y)\vee d(y,z)\geq d(x,z)$. (As for generalised metrics, also generalised ultrametrics can be symmetric and/or skeletal.) Because for any $a,b\in[0,\infty]\op$ we obviously have $a+b\geq a\vee b$, the identity function is a lax morphism from $([0,\infty]\op,\vee,0)$ to $([0,\infty]\op,+,0)$; the induced change of base functor is simply the inclusion of generalised ultrametric spaces into generalised metric spaces.
\end{example}
In Example \ref{13} and further we shall come back to this example.

\section{Partial metric spaces as enriched categories}\label{B}

\subsection{Diagonals}

It often happens in practice that quantaloids arise from quantales by one or another universal construction. We shall describe one such case, which will turn out to be crucial to describe the categorical content of partial metric spaces.

First we recall a definition from \cite{stu13}:
\begin{definition}\label{1}
Fixing two morphisms $f\colon A\to B$ and $g\colon C\to D$ in a quantaloid $\Q$, we say that a third morphism $d\colon A\to D$ in $\Q$ is a {\em diagonal from $f$ to $g$} if any (and thus both) of the following equivalent conditions holds:
\begin{enumerate}
\item\label{a} there exist $x\colon A\to C$ and $y\colon B\to D$ in $\Q$ such that $y\circ f=d=g\circ x$,
\item\label{b} $g\circ(g\se d)=d=(d\sw f)\circ f$.
\end{enumerate}
\end{definition}
\begin{proof}[Proof of the equivalence.]
Obviously $(\ref{b}\Rightarrow\ref{a})$ is trivial. Conversely, $d=g\circ x\leq g\circ (g\searrow d)\leq d$ holds because $g\circ x=d$ implies $x\leq g\searrow d$; similarly $d=(d\swarrow f)\circ f$ follows from $y\circ f=d$.
\end{proof}
The reason for the term ``diagonal'' is clear from a picture to accompany the first condition in the above definition: given the solid morphisms in the diagram
$$\xymatrix@=8ex{
A\ar[d]_f\ar[dr]^d\ar@{.>}[r]^x & C\ar[d]^g \\
B\ar@{.>}[r]_y & D}$$
in $\Q$, one seeks to add the dotted morphisms, to form a commutative diagram. The equivalent second condition then adds that, whenever such $x$ and $y$ exist, then there is a {\em canonical choice} for them, namely $x=g\se d$ and $y=d\sw f$.
\begin{proposition}[\cite{stu13}]\label{1.1}
For any (small) quantaloid $\Q$, a new (small) quantaloid $\mathcal{D(Q)}$ of {\em diagonals in $\Q$} is built as follows: 
\begin{itemize}
\item the objects of $\D(\Q)$ are the morphisms of $\Q$,
\item a morphism from $f$ to $g$ in $\D(\Q)$ is a diagonal from $f$ to $g$ in $\Q$,
\item the composition of two diagonals $d\colon f\to g$ and $e\colon g\to h$ is defined to be
$$e\circ_g d:=(e\sw g)\circ g \circ (g\se d),$$ 
\item the identity on $f$ is $f\colon f\to f$ itself,
\item and the supremum of a set of diagonals $(d_i\colon f\to g)_{i\in I}$ is computed ``as in $\Q$''.
\end{itemize}
\end{proposition}
\begin{remark}\label{1.x}
Regarding composition of diagonals, it is useful to point out that the formula given above for $e\circ_g d$ is really just one of many equivalent expressions for the composite arrow from the upper left corner to the lower right corner in the following commutative diagram:
$$\xymatrix@=8ex{
\cdot\ar[d]_f\ar[dr]|{\ d\ }\ar@{.>}[r]^x & \cdot\ar[d]_g\ar[dr]|{\ e\ }\ar@{.>}[r]^u & \cdot\ar[d]^h \\
\cdot\ar@{.>}[r]_y & \cdot\ar@{.>}[r]_v & \cdot}$$
Particularly, in doing so, one can {\em choose any $x,y,u,v$ that make the diagram commute}, not just the canonical $x=g\se d$, $y=d\sw f$, $u=h\se e$, $v=e\sw g$.
\end{remark}
There is an obvious full and faithful inclusion {\em homomorphism}\footnote{Better still, this embedding enjoys a powerful universal property: it is the {\em splitting-of-everything} in $\Q$; and consequently it is the unit of a 2-monad on the category $\Quant$ of small quantaloids. This has been described by Grandis \cite{gra02} for small categories.} of $\Q$ in $\D(\Q)$: 
\begin{equation}\label{7.1}
I\:\Q\to\D(\Q)\:\Big(f\:A\to B\Big)\mapsto\Big(f\:1_A\to 1_B\Big).
\end{equation}
It is thus a natural problem to study how properties of a given quantaloid $\Q$ extend (or not) to the larger quantaloid $\D(\Q)$. For later use we record a simple example:
\begin{example}\label{7.2}
Say that a quantaloid $\Q$ is {\it symmetric} whenever the identity function $\Q_1\to\Q_1$ is an involution on $\Q$; explicitly, this means that $\Q(A,B)=\Q(B,A)$ and $f\circ g=g\circ f$ for all objects $A,B$ and all morphisms $f,g$ of $\Q$. It is then a simple fact that $\Q$ is symmetric if and only if $\D(\Q)$ is. Note that a `symmetric quantaloid with a single object' is precisely a commutative quantale. So as a particular case we find here that, for any commutative quantale $Q$, the quantaloid $\D(Q)$ is symmetric.
\end{example}
In the next subsection we shall study a particular class of quantales and quantaloids -- the so-called {\it divisible} ones -- whose diagonals behave particularly well; thereafter we shall be interested in categories enriched in $\D(Q)$ whenever $Q$ is a divisible quantale.

\subsection{Divisible quantaloids}

In \cite{stu13} we first introduced our notion of {\it divisible quantaloid}, but that first definition contained some redundancies---so here we spell it out again, in a more optimal form.
\begin{definition}\label{10}
A quantaloid $\Q$ is {\em divisible}\footnote{Note how the notion of {\em divisibility} is self-dual: $\Q$ is divisible if and only if $\Q\op$ is. Put differently, formally it makes sense to define $\Q$ to be {\em semi-divisible} if, for all $d,e\:A\to B$ in $\Q$, $e\circ(e\se d)=d\wedge e$; and then $\Q$ is divisible if and only if both $\Q$ and $\Q\op$ are semi-divisible. 
} if it satisfies any (and thus all) of the following equivalent conditions:
\begin{enumerate}
\item\label{10.3.1} for all $d,e\:A\to B$ in $\Q$: $d\leq e$ if and only if there exist $x\:A\to A$ and $y\:B\to B$ such that $e\circ x=d=y\circ e$,
\item\label{10.3} for all $d,e\:A\to B$ in $\Q$: $d\leq e$ if and only if $e\circ(e\searrow d)=d=(d\swarrow e)\circ e$,
\item\label{10.1} for all $d,e\:A\to B$ in $\Q$: $e\circ (e\searrow d)=d\wedge e=(d\swarrow e)\circ e$.
\item\label{10.4} for all $e\:A\to B$ in $\Q$: $\D(\Q)(e,e)=\ \downarrow\! e$ (as sublattices of $\Q(A,B)$),
\item\label{10.5} for all $d,e\:A\to B$ in $\Q$: $\D(\Q)(d,e)=\ \downarrow\! (d\wedge e)$ (as sublattices of $\Q(A,B)$).
\end{enumerate}
\end{definition}
\begin{proof}[Proof of the equivalences.]
The implications $(\ref{10.3.1}\Leftarrow\ref{10.3}\Leftarrow\ref{10.1})$ hold trivially, as do ($\ref{10.3}\Leftrightarrow\ref{10.4}\Leftarrow\ref{10.5})$; and to see that $(\ref{10.3.1}\Rightarrow\ref{10.3})$ one merely needs to adapt slightly the argument given for the equivalence of the conditions in Definition \ref{1}. 

Now assume $(\ref{10.3})$; putting $e=1_A$ in the condition and using that $\Q(A,A)=\D(\Q)(1_A,1_A)$, shows that $\Q$ is necessarily integral (meaning that, for every object $A$, the identity morphism $1_A$ is the top element of $\Q(A,A)$). From $d\wedge e\leq e$ we get $d\wedge e=e\circ(e\se(d\wedge e))$; but $e\se(d\wedge e)=(e\se d)\wedge(e\se e)$ (because right adjoints preserve infima) and $e\se e=1_A$ (by integrality of $\Q$) so $(e\se d)\wedge(e\se e)=(e\se d)$; altogether, we find $d\wedge e=e\circ(e\se d)$. A similar computation proves that $d\wedge e=(d\sw e)\circ e$, so in all we proved $(\ref{10.3}\Rightarrow\ref{10.1})$.

To complete the proof we shall show that $(\ref{10.3}\Rightarrow\ref{10.5})$. First, we use again that $\Q$ is necessarily integral to find, for any $f\in\D(\Q)(d,e)$, that $f=(f\sw d)\circ d\leq 1_B\circ d=d$; and similar for $f\leq e$. Conversely, if $f\leq d\wedge e$ then $e\circ(e\se f)=f$ and $(f\sw d)\circ d=f$ both follow directly from the assumption, and show that $f\in\D(\Q)(d,e)$.
\end{proof}
To stress that the definition in \cite{stu13} agrees with Definition \ref{10} above, we record an observation made in the proof above\footnote{This Proposition \ref{11}, and also Proposition \ref{11.1}, help to show that there are very many {\it non-divisible} quantaloids. For instance, if $(M,\circ,1)$ is any monoid, then the free quantale $(\P(M),\circ,\{1\})$ is divisible if and only if $M=\{1\}$. The quantaloid $\Dist(\Q)$ is not divisible in general, even for divisible $\Q$---because it is not integral. The quantale of sup-endomorphisms on a sup-lattice is not divisible in general. And so forth.}:
\begin{proposition}\label{10.x} 
A divisible quantaloid $\Q$ is always integral.
\end{proposition}
We also observe:
\begin{proposition}\label{11}
A divisible quantaloid $\Q$ is always locally localic\footnote{This means that each lattice $\Q(A,B)$ of morphisms in $\Q$ with the same domain and the same codomain is a locale. This does {\em not} mean that composition in $\Q$ preserves finite infima in each variable (e.g.\ the empty infimum is hardly ever preserved).}.
\end{proposition} 
\begin{proof}
For $(f_i)_{i\in I},g$ be in $\Q(A,B)$ we certainly have $\bigvee_i(f_i\wedge g)\leq(\bigvee_if_i)\wedge g$. Assuming $\Q$ to be divisible, we can furthermore compute that
$$(\bigvee_if_i)\wedge g=(\bigvee_if_i)\circ\Big((\bigvee_jf_j)\se g\Big)=\bigvee_i\Big(f_i\circ\Big(\bigwedge_j(f_j\se g)\Big)\Big)\leq\bigvee_i\Big(f_i\circ(f_i\se g)\Big)=\bigvee_i(f_i\wedge g)$$
which leads to the conclusion: $\bigvee_i(f_i\wedge g)=(\bigvee_if_i)\wedge g$.
\end{proof}
The very definition of divisibility already shows a link with the diagonal construction; the next proposition adds to that:
\begin{proposition}\label{11.1}
A quantaloid $\Q$ is divisible if and only if $\D(\Q)$ is divisible.
\end{proposition}
\begin{proof}
As we may regard $\Q$ as a full subquantaloid of $\D(\Q)$, if the latter is divisible then so must be the former.

Now suppose that $\Q$ is divisible. By Condition $(\ref{10.4})$ in Definition \ref{10} we find that $\D(\Q)$ is integral, so one implication in Condition $(\ref{10.3.1})$ is trivial for $\D(\Q)$. For the other implication, consider two diagonals
$$\xymatrix@=8ex{
X_0\ar[d]_f\ar@<1ex>[dr]^d\ar@<-1ex>[dr]_{d'} & Y_0\ar[d]^g \\
X_1 & Y_1}$$
such that $d\leq d'$ in $\D(\Q)(f,g)$. Because we necessarily have $d\leq d'$ in $\Q(X_0,Y_1)$ too, it follows from the assumptions on $\Q$ that there exist $x\:X_0\to X_0$ and $y\:Y_1\to Y_1$ such that $y\circ d'=d=d'\circ x$ in $\Q$. Furthermore, $f\circ x\leq f$ in $\Q(X_0,X_1)$ (because $\Q$ is integral), so that there exists an $x'\:X_1\to X_1$ such that $\xi:=f\circ x=x'\circ f$; and for similar reasons there is an $y'\:Y_0\to Y_0$ such that $\eta:=y\circ g=g\circ y'$. Displaying all these morphisms in the diagram
$$\xymatrix@=8ex{
X_0\ar[d]_f\ar@{.>}[r]^x\ar[dr]|{\xi} & X_0\ar[d]_f\ar@<1ex>[dr]^d\ar@<-1ex>[dr]_{d'} & Y_0\ar[dr]|{\eta}\ar[d]^g\ar@{.>}[r]^{y'} & Y_0\ar[d]^g \\
X_1\ar@{.>}[r]_{x'} & X_1 & Y_1\ar@{.>}[r]_y & Y_1}$$
shows that $\xi$ and $\eta$ are diagonals too, and from Remark \ref{1.x} it is clear that $d'\circ_f\xi=d=\eta\circ_gd'$. So we showed that $d'$ divides $d$ in $\D(\Q)$, as required.
\end{proof}
An important class of examples is provided by:
\begin{example}\label{12}
Any locale $H$ is (commutative and) divisible: for any $a,b\in H$ we have $a\wedge(a\Rightarrow b)\leq b$ by the universal property of the ``implication'', and $a\wedge(a\Rightarrow b)\leq a$ holds trivially, so we already find $a\wedge(a\Rightarrow b)\leq a\wedge b$; and conversely, $b\leq (a\Rightarrow b)$ holds by its equivalence to the trivial $a\wedge b\leq b$, and therefore also $a\wedge b\leq a\wedge(a\Rightarrow b)$ holds. Via the argument in Example \ref{7.2} and the above Proposition \ref{11.1} it follows that also the quantaloid $\D(H)$ is symmetric and divisible\footnote{Note that -- because every element of $H$ is idempotent -- the quantaloid $\D(H)$ is exactly the universal {\it splitting-of-idempotents} in $H$; that is to say, any homomorphism $H\to\R$ into a quantaloid in which all idempotents split, extends essentially uniquely to a homomorphism $\D(H)\to\R$.}. Both $H$ and $\D(H)$ are Cauchy-bilateral (for the identity involution), see \cite[Example 4.5]{heystu11}. In particular is this all true for $H=([0,\infty]\op,\vee,0)$.
\end{example}
We hasten to point out our {\it other} main example:
\begin{example}\label{13}
The Lawvere quantale $R=([0,\infty]\op,+,0)$ is (commutative and) divisible. The ``implication'' in this  quantale is given by $a\sq b=0 \vee (b-a)$ (truncated substraction), and so it is easily seen that $a+(a\sq b)=a\vee b$, as required. It thus follows that its quantaloid of diagonals $\D(R)$ is symmetric and divisible. In \cite[Example 4.4]{heystu11} it is shown that $R$ is Cauchy-bilateral; we shall now prove
the stronger fact that also $\D(R)$ is Cauchy-bilateral. 

Because the Lawvere quantale is divisible, we know by Condition $(\ref{10.4})$ in Definition \ref{10} that its quantaloid of diagonals is integral. Therefore, as explained in \cite[Definition 4.2]{heystu11}, the latter is Cauchy-bilateral if and only if the following holds\footnote{This expresses exactly that, whenever $f_i\:a\to b_i$ and $g_i\:b_i\to a$ are diagonals so that the supremum of the composites $g_i\circ_{b_i}f_i$ is bigger than the identity diagonal on $a$, then so is the supremum of the composites of $(g_i\vee f_i\o)\circ_{b_i}(g_i\o\vee f_i)$; but the involution on $\D(Q)$ is the identity -- stemming from $Q$'s commutativity -- and composition is computed as $g\circ_b f=g+f-b$.}: for any index-set $I$ and elements $a,b_i,f_i,g_i\in[0,\infty]$,
$$\mbox{ if }f_i\wedge g_i\geq a\vee b_i\mbox{ and }a\geq\bigwedge_i(g_i+f_i-b_i)\mbox{ then }a\geq\bigwedge_i(2(f_i\vee g_i)-b_i).$$
Because $R$ is linearly ordered and the formulas are symmetric in $f_i$'s and $g_i$'s, we may suppose that $g_i\geq f_i\geq b_i$ for all $i\in I$. Under this harmless extra assumption we can compute that
$$\bigwedge_i g_i\geq\bigwedge_i f_i\geq a\geq \bigwedge_i(g_i+f_i-b_i)\geq\bigwedge_i g_i$$
and furthermore
$$\bigwedge_i(2(f_i\vee g_i)-b_i)=\bigwedge_i(2g_i-b_i)\geq\bigwedge_ig_i\geq\bigwedge_if_i.$$
It is thus sufficient to prove that: for any index-set $I$ and elements $b_i,f_i,g_i\in[0,\infty]$,
$$\mbox{ if }g_i\geq f_i\geq b_i\mbox{ and }\bigwedge_i f_i\geq\bigwedge_i(g_i+f_i-b_i)\mbox{ then }\bigwedge_if_i\geq\bigwedge_i(2g_i-b_i).$$
But from $\bigwedge_i(g_i+f_i-b_i)\leq\bigwedge_if_i$ we know that for any $\varepsilon>0$ there exists $k\in I$ such that for any $j\in I$: $g_k+f_k-b_k<f_j+\varepsilon$; and upon putting $j=k$ it follows that $g_k-b_k<\varepsilon$. Secondly, since $\bigwedge_i g_i\leq\bigwedge_i f_i$ we find for any $\varepsilon>0$ some  $i\in I$ such that for any $j\in I$, $g_i<f_j+\varepsilon$. Summing up, for any $\varepsilon>0$ there is an $i\in I$ such that for any $j\in I$, $2g_i-b_i<f_j+2\varepsilon$; and this means exactly that $\bigwedge_i(2g_i-b_i)\leq\bigwedge_if_i$, as wanted.
\end{example}
The close relationship between the two previous examples, $([0,\infty]\op,\vee,0)$ and $([0,\infty]\op,+,0)$, can be traced back to divisibility, as follows.

Let $Q=(Q,\circ,1)$ be any divisible quantale, and write $Q_H=(Q,\wedge,1)$ for the underlying locale; because $Q$ is integral it follows that the identity function is a lax morphism from $Q_H$ to $Q$. We must distinguish between the quantaloid $\D(Q)$ of diagonals in $Q$ and the quantaloid $\D(Q_H)$ of diagonals in $Q_H$. However, both these divisible quantaloids have the same objects, and -- as spelled out above -- for fixed $f,g\in Q$ we also find that
$$\D(Q)(f,g)=\ \downarrow\!(f\wedge g)=\D(Q_H)(f,g).$$
Furthermore, the identity on an object $f\in Q$, in both $\D(Q)$ and $\D(Q_H)$, is the greatest element of $\D(Q)=\D(Q_H)$, {\it viz.}\ $f$ itself. So the only (but crucial) difference between both these quantaloids, is the composition law:
\begin{itemize}
\item the composite of $d\:f\to g$ and $e\:g\to h$ in $\D(Q)$ is $e\circ_g d=(e\sw g)\circ g\circ(g\se d)$,
\item the composite of $d\:f\to g$ and $e\:g\to h$ in $\D(Q_H)$ is $e\wedge d$.
\end{itemize}
However, these expressions compare: because $e\circ_g d=e\circ(g\se d)\leq e$, and similarly $e\circ_g d\leq d$, so $e\circ_g d\leq e\wedge d$. In other words, the identity function on objects and arrows defines a normal lax morphism from $\D(Q_H)$ to $\D(Q)$. Furthermore, the lax morphisms $Q_H\to Q$ and $\D(Q_H)\to\D(Q)$ commute with the full embeddings $Q\to\D(Q)$ and $Q_H\to\D(Q_H)$ to make the following square commute:
\begin{equation}\label{14}
\begin{array}{c}
$$\xymatrix@!=8ex{
Q_H\ar[r]\ar[d] & Q\ar[d] \\
\D(Q_H)\ar[r] & \D(Q)}
\end{array}
\end{equation}

When applying the above constructions to $R=([0,\infty]\op,+,0)$, we already know that categories enriched in $R$ are generalised metric spaces (Example \ref{d}); we also know that categories enriched in $R_H$ are generalised ultrametric spaces and that the change of base induced by the lax morphism from $R_H$ to $R$ encodes precisely the inclusion of ultrametrics into metrics (Example \ref{s}). In the next section we study the two {\it other} bases of enrichement made available in Diagram (\ref{14}).

\subsection{Partial categories, partial metrics}\label{yz}

When $\Q$ is a small quantaloid, then so is $\D(\Q)$; hence the theory of enriched categories applies to $\D(\Q)$ as much as it does to $\Q$. The full embedding $I\:\Q\to\D(\Q)$ induces a change of base $\Cat(\Q)\to\Cat(\D(\Q))$
which shows how $\Q$-categories fit into $\D(\Q)$-categories. For the sake of exposition, we introduce the following terminology and notation:
\begin{definition}\label{20}
A {\bf partial $\Q$-enriched category (functor, distributor)} is a $\D(\Q)$-enrich\-ed category (functor, distributor). We write $\Par\Cat(\Q):=\Cat(\D(\Q))$ and $\Par\Dist(\Q):=\Dist(\D(\Q))$.
\end{definition}
Explicitly, a partial $\Q$-category $\bbC$ consists of
\begin{itemize}\setlength{\itemindent}{3em}
\item[{\tt (obj)}] a set $\bbC_0$,
\item[{\tt (typ)}] a function $t\:\bbC_0\to\Q_1$,
\item[{\tt (hom)}] a function $\bbC\:\bbC_0\times\bbC_0\to\Q_1$,
\end{itemize}
such that, in the quantaloid $\Q$, we have that
\begin{itemize}\setlength{\itemindent}{3em}
\item[{\tt (PC0)}] $\bbC(y,x)$ is a diagonal from $tx$ to $ty$: $(\bbC(y,x)\sw tx)\circ tx=\bbC(y,x)=ty\circ(ty\se\bbC(y,x))$,
\item[{\tt (PC1)}] $tx\leq\bbC(x,x)$,
\item[{\tt (PC2)}]\hspace{-1.5ex}\footnote{Or any of the equivalent expressions obtained by replacing the left hand side, thanks to {\tt (PC0)}, with either $(\bbC(z,y)\swarrow ty)\circ\bbC(y,x)$ or $\bbC(z,y)\circ(ty\searrow\bbC(y,x)$.} $(\bbC(z,y)\swarrow ty)\circ ty\circ(ty\searrow\bbC(y,x))\leq\bbC(z,x).$
\end{itemize}
Similarly one can express the notions of $\D(\Q)$-enriched functor and distributor to avoid explicit references to the diagonal construction, and speak of `partial $\Q$-functor' and `partial $\Q$-distributor' between partial $\Q$-categories.

Upon identification of $\Q$ with its image in $\D(\Q)$ along the full embedding $I\:\Q\to\D(\Q)$, it is clear that (``total'') $\Q$-categories (and functors between them) are exactly the same thing as a partial $\Q$-categories for which all object-types are identity morphisms (and partial functors between them). Indeed, the change of base $\Cat(\Q)\to\Par\Cat(\Q)$ induced by the full embedding $I\:\Q\to\D(\Q)$ is precisely the full inclusion of $\Q$-categories (and functors) into partial $\Q$-categories (and partial functors). 

As a converse to the inclusion of $\Q$ into $\D(\Q)$, we can observe that any diagonal $d\:f\to g$ can be ``projected'' onto its ``domain'' and onto its ``codomain'':
\begin{proposition}\label{7}
For any quantaloid $\Q$, both
$$J_0\Big(d\:f\to g\Big)=\Big(g\se d\:\dom(f)\to\dom(g)\Big)$$
$$J_1\Big(d\:f\to g\Big)=\Big(d\sw f\:\cod(f)\to\cod(g)\Big)$$ 
are lax morphisms. The induced change of base functors $J_0,J_1\:\Par\Cat(\Q)\to\Cat(\Q)$, send a partial $\Q$-category $\bbC$ to:
\begin{itemize}
\item the $\Q$-category $J_0\bbC$ with object set $\bbC_0$, type function $\bbC_0\to\Q_0\:x\mapsto\dom(tx)$, and hom function $\bbC_0\times\bbC_0\to\Q_1\:(y,x)\mapsto ty\se\bbC(y,x)$,
\item the $\Q$-category $J_1\bbC$ with object set $\bbC_0$, type function $\bbC_0\to\Q_0\:x\mapsto\cod(tx)$, and hom function $\bbC_0\times\bbC_0\to\Q_1\:(y,x)\mapsto\bbC(y,x)\sw tx$.
\end{itemize}
\end{proposition}
\begin{proof}
For morphisms $f,g$ in $\Q$, the map $\D(\Q)(f,g)\to\Q(\dom(f),\dom(g))\:d\mapsto g\se d$ preserves order. Given diagonals $d\:f\to g$ and $e\:g\to h$, we know that $h\circ (h\se e)\circ(g\se d)=e\circ_g d$, from which it follows by lifting through $h$ that $(h\se e)\circ(g\se d)\leq h\se (e\circ_g d)$, or in other words, $J_0(e)\circ J_0(d)\leq J_0(e\circ_g d)$.
Finally, for any morphism $f$ in $\Q$ we have that $1_{\dom(f)}\leq (f\se f)=J_0(1_f)$. The proof for $J_1$ is entirely dual.
\end{proof}
In a somewhat different context \cite{taolaizha12}, these change of base functors have been called the {\it forward} and {\it backward globalisation} of a partial $\Q$-category. It can be remarked that, since $J_0\:\D(\Q)\to\Q$ is a left inverse to $I\:\Q\to\D(\Q)$ (that is, $J_0\circ I$ is the identity on $\Q$), the same is true for the induced functors $J_0\:\Par\Cat(\Q)\to\Cat(\Q)$ and $I\:\Cat(\Q)\to\Par\Cat(\Q)$ (and similar for $J_1$).

Even though it could be an interesting topic to compare partial $\Q$-categories with ``total'' $\Q$-categories for a general base quantaloid $\Q$, we shall narrow our study down to a more specific situation: in the rest of this section we shall be concerned only with {\it commutative and divisible quantales}---in keeping with our main example, the Lawvere quantale $R=([0,\infty]\op,+,0)$. 

Let us first note that, whenever $Q=(Q,\circ,1)$ is a commutative quantale, the function $Q\times Q\to Q\:(f,g)\mapsto f\circ g$ is a homomorphism of quantales, so that composition with the lax morphism $(J_0,J_1)\:\D(Q)\to Q\times Q$ (whose components $J_0$ and $J_1$ are those of Proposition \ref{7}) produces yet another lax morphism from $\D(\Q)$ to $\Q$:
\begin{proposition}\label{9}
If $Q$ is a commutative quantale\footnote{The commutativity of the multiplication implies, by uniqueness of adjoints, that liftings and extensions are the same thing; so in this case we shall write $x\sq y$ instead of $x\se y=y\sw x$. We reserve the notation $x\Rightarrow y$ for the case where the multiplication is given by binary infimum, i.e.\ when the quantale considered is actually a locale.}, then 
$$K\:\D(Q)\to Q\:(d\:f\to g)\mapsto((g\sq d)\circ(f\sq d))$$ 
is a lax morphism. The induced change of base $K\:\Par\Cat(Q)\to\Cat(Q)$ sends a partial $Q$-category $\bbC$ to the symmetric $Q$-category $K\bbC$ with object set $\bbC_0$ and hom function $\bbC_0\times\bbC_0\to Q\:(y,x)\mapsto(ty\sq\bbC(y,x))\circ(tx\sq\bbC(y,x))$.
\end{proposition}
Unlike $J_0$ and $J_1$, the lax morphism $K$ is {\it not} a left (or right) inverse to $I\:Q\to\D(Q)$.

Secondly, let us narrow down the definition of partial $\Q$-category \cite{stu13}:
\begin{proposition}\label{15}
If $Q=(Q,\bigvee,\circ,1)$ is a divisible quantale, then a partial $Q$-category $\bbC$ is determined by a set $\bbC_0$ together with a function $\bbC\:\bbC_0\times\bbC_0\to Q\:(y,x)\mapsto\bbC(y,x)$ satisfying
$$\bbC(y,x)\leq\bbC(x,x)\wedge\bbC(y,y)\mbox{ and }(\bbC(z,y)\swarrow\bbC(y,y))\circ\bbC(y,x)\leq\bbC(z,x).$$
A partial functor $F\:\bbC\to\bbD$ between partial $Q$-categories is a function $F\:\bbC_0\to\bbD_0$ satisfying
$$\bbC(x,x)=\bbD(Fx,Fx)\mbox{ and }\bbC(y,x)\leq\bbD(Fy,Fx).$$
And a partial distributor $\Phi\:\bbC\dist\bbD$ is a function $\Phi\:\bbD_0\times\bbC_0\to Q$ satisfying
$$\Phi(y,x)\leq\bbC(x,x)\wedge\bbD(y,y),\quad (\bbD(y',y)\swarrow\bbD(y,y))\circ\Phi(y,x)\leq\Phi(y',x)$$
$$\mbox{ and }\quad
\Phi(y,x)\circ(\bbC(x,x)\searrow\bbC(x,x'))\leq\Phi(y,x').$$
\end{proposition}
\begin{proof}[Sketch of proof.]
Take the explicit description, below Definition \ref{20}, of a partial $\Q$-category $\bbC$, and weed out the redundancies due to the particularities of the divisible quantale $Q$ (and the therefore also divisible quantaloid $\D(Q)$): because both the set of objects and the set of arrows of $\D(Q)$ are equal to $Q$, we find that the both the type function and the hom function take values in $Q$; $\D(Q)$ is integral and $tx=1_{tx}$ by construction, so the reflexivity of the hom function becomes $tx=\bbC(x,x)$, making the type function implicit in the hom function and {\tt [PC1]} obsolete; divisibility of $Q$ makes {\tt [PC0]} equivalent to $\bbC(x,y)\leq\bbC(x,x)\wedge\bbC(y,y)$; and formulating the composition in $\D(Q)$ back into terms proper to $Q$, {\tt [PC2]} is exactly $(\bbC(z,y)\swarrow\bbC(y,y))\circ\bbC(y,x)\leq\bbC(z,x)$. Similar simplifications apply to functors and distributors.
\end{proof}

Finally, we can fully develop -- as we set out to do -- the notion of `partial metric space':
\begin{example}\label{16}
For Lawvere's quantale $R=([0,\infty]\op,+,0)$, and adopting common notations, a partial $R$-category $\bbX$ is precisely a set $X:=\bbX_0$ together with a function $p:=\bbX:X\times X\to[0,\infty]$
satisfying
$$p(y,x)\geq p(x,x)\vee p(y,y)\mbox{ and }p(z,y)-p(y,y)+p(y,x)\geq p(z,x).$$
In line with Example \ref{d} we call such a structure $(X,p)$ a {\it generalised partial metric space}---indeed, upon imposing {\em finiteness}, {\em symmetry} and {\em separatedness}, we recover exactly the partial metric spaces of \cite{mat94}, whose definition we recalled in the Introduction. A partial functor $f\:(X,p)\to(Y,q)$ between such spaces is a non-expansive map $f\:X\to Y\:x\mapsto fx$ satisfying furthermore $p(x,x)=q(fx,fx)$; these objects and morphisms thus form the (locally ordered) category $\Par\Met:=\Par\Cat(R)=\Cat(\D(R))$.

Furthermore, the underlying locale $R_H=([0,\infty]\op,\vee,0)$ of the Lawvere quantale is also a divisible quantale. A partial $R_H$-enriched category $\bbX$ is a set $X:=\bbX_0$ together with a function $u:=\bbX\:X\times X\to[0,\infty]$
satisfying
$$\mbox{$u(y,x)\geq u(x,x)\vee u(y,y)$ \ and \ $u(z,y)\vee u(y,x)\geq u(z,x)$}.$$
For all the obvious reasons we shall call such a $(X,u)$ a {\it generalised partial ultrametric space}. These spaces are the objects of a locally ordered) category $\Gen\Par\Ult\Met:=\Par\Cat(R_H)=\Cat(\D(R_H))$. 
\end{example}

The commutative Diagram \eqref{14} of lax morphisms induces a commutative diagram
$$\xymatrix@=8ex{
\Gen\Ult\Met\ar[r]\ar[d] & \Gen\Met\ar[d] \\
\Gen\Par\Ult\Met\ar[r] & \Gen\Par\Met}$$
in which all arrows are full embeddings. When restricting to symmetric, finite and separating distance functions in all four categories in this square, one finds the appropriate categories of ``non-generalised'' (partial) (ultra)metric spaces.

On the other hand, as a corollary of Propositions \ref{7} and \ref{9} (which apply to $R$ as well as $R_H$!), we have three ways to compute a ``total'' (generalised) (ultra)metric from a partial one: given $(X,p)$ we find
\begin{itemize}
\item $p_0(y,x):=p(y,x)-p(y,y)$ via the lax morphism $J_0\:\D(R)\to R$,
\item $p_1(y,x):=p(y,x)-p(x,x)$ via the lax morphism $J_1\:\D(R)\to R$,
\item $p_K(y,x):=2p(y,x)-p(x,x)-p(y,y)$ via the lax morphism $K\:\D(R)\to R$.
\end{itemize}
These constructions will be useful in the next Section.

To end this Section, we insist on the fact that partial functors between (generalised) partial (ultra)metrics are non-expansive maps {\it that preserve self-distance}. At first sight this may seem too strong a requirement---would it not be more natural to allow (non-expansive) functions $f\:(X,p)\to (Y,q)$ to {\it decrease} the self-distances too? But for our later purposes (namely, to canonically associate a topology to every quantaloid-enriched category, and therefore also to each partial metric space) the latter type of map is not suitable (it does not give rise to continuous maps). However, there is also a simple algebraic argument in favour of maps that do not decrease self-distances (apart from their origin as functors in the appropriate categorical setting, {\it viz.}~as $\D(R)$-enriched functors). Consider the one-element partial metric space $\1_a$ whose single element has self-distance $a\in[0,\infty]$. General non-expansive maps $f\:\1_a\to(X,p)$ are in 1-1 correspondence with elements of $X$ whose self-distance is {\it at most} $a$; if we impose $f$ to preserve self-distance, then it picks out an element of $X$ whose self-distance is {\it exactly} $a$. The second situation is thus to be preferred, if one wants to be able to identify each element of $(X,p)$ with {\it precisely one} map defined on a singleton partial metric.

\section{Topology from enrichment}\label{C}

\subsection{Density and closure}\label{20x}

A functor $F\:\bbC\to\bbD$ between $\Q$-categories is fully faithful when $\bbC(y,x)=\bbD(Fy,Fx)$ for every $x\in\bbC_0$ and $y\in\bbD$; equivalently, this says that the unit of the adjunction of distributors $F_*\dashv F^*$ is an equality (instead of a mere inequality). The complementary notion to fully faithfulness will be of importance to us in this section: 
\begin{definition}\label{21}
A functor $F\:\bbC\to\bbD$ between $\Q$-categories is {\bf fully dense} if the counit of the adjunction of distributors $F_*\dashv F^*$ is an equality (instead of a mere inequality); explicitly, we have for all $x,y\in\bbD_0$ that
$$\bbD(y,x)=\bigvee_{c\in\bbC_0}\bbD(y,Fc)\circ\bbD(Fc,x).$$
\end{definition}
It is clear that an {\it essentially surjective} $F\:\bbC\to\bbD$ (meaning that for every $y\in\bbD$ there exists an $x\in\bbC$ such that $Fx\cong y$) is always fully dense; but the converse need not hold.
\begin{proposition}\label{22}
A functor $F\:\bbC\to\bbD$ between $\Q$-categories is fully dense if and only if it is essentially epimorphic, i.e.\ for every $H,K\:\bbD\to\bbE$, if $H\circ F\cong K\circ F$ then $H\cong K$.
\end{proposition}
\begin{proof}
If $F$ is fully dense and $H\circ F\cong K\circ F$, then -- looking at the represented distributors -- we can precompose both sides of $\bbE(-,H-)\tensor\bbD(-,F-)=\bbE(-,K-)\tensor\bbD(-,F-)$ with $\bbD(F-,-)$ to find $\bbE(-,H-)=\bbE(-,K-)$, which means precisely that $H\cong K$.

To see the converse, consider the Yoneda embedding $Y_{\bbD}\:\bbD\to\P\bbD\:d\mapsto\bbD(-,d)$ alongside the functor $Z\:\bbD\to\P\bbD\:d\mapsto\bbD(-,F-)\tensor\bbD(F-,d)$. Because $Y_{\bbD}(Fc)=Z(Fc)$ holds for all $c\in\bbC$, the assumed essential epimorphic $F$ provides that $Y_{\bbD}d\cong Zd$ for all $d\in\bbD$---but since $\P\bbD$ is a skeletal $\Q$-category (isomorphic objects are necessarily equal), we actually have that $Y_{\bbD}d=Zd$ for all $d\in\bbD$. This says precisely that $F$ is fully dense.
\end{proof}

Whenever $\bbC$ is a $\Q$-category, any $S\subseteq\bbC_0$ determines a full subcategory $\bbS\hookto\bbC$. In particular, two subsets $S\subseteq T\subseteq\bbC_0$ determine an inclusion of full subcategories $\bbS\hookto\bbT\hookto\bbC$. Slightly abusing terminology we shall say that $S$ is fully dense in $T$ whenever the canonical inclusion $\bbS\hookto\bbT$ is fully dense. Fixing $S$, we now want to compute the largest $T$ in which $S$ is fully dense. 
\begin{lemma}\label{23.1}
If subsets $S,(T)_{i\in I}$ of $\bbC_0$ are such that $S$ is fully dense in each $T_i$, then $S$ is fully dense in $\bigcup_iT_i$.
\end{lemma}
\begin{proof}
Let us write respectively $\bbS$, $\bbT_i$ and $\bbT$ for the full subcategories of $\bbC$ determined by $S\subseteq\bbC_0$, $T_i\subseteq\bbC_0$ and $\bigcup_iT_i\subseteq\bbC_0$. Suppose that functors $F,G\:\bbT\to\bbD$ agree (to within isomorphism) on $S$, then density of $S$ in each $T_i$ makes them agree on each $T_i$, and therefore on $\bigcup_iT_i$. That is, the inclusion of $S$ in $\bigcup_iT_i$ is fully dense, according to Proposition \ref{22}.
\end{proof}
The above lemma allows for the following definition:
\begin{definition}\label{23}
Let $\bbC$ be a $\Q$-category. The {\bf categorical closure} of a subset $S\subseteq\bbC_0$ is the largest subset $\cl{S}\subseteq\bbC_0$ in which $S$ is fully dense; that is to say, 
$$\cl{S}=\bigcup\{T\subseteq\bbC_0\mid\mbox{$S$ is fully dense in $T$}\}.$$
\end{definition}
To explicitly compute the closure of a subset $S$ of objects of $\bbC$, we can use:
\begin{proposition}\label{24}
Let $\bbC$ be a $\Q$-category and for $S\subseteq\bbC_0$ write $i\:\bbS\hookto\bbC$ for the corresponding full embedding. For an object $x\in\bbC$ the following are equivalent:
\begin{enumerate}
\item\label{aa} $x\in\cl{S}$,
\item\label{bb} $\bbC(i-,x)\dashv\bbC(x,i-)$, or explicitly: $1_{tx}\leq\bigvee_{s\in S}\bbC(x,s)\circ\bbC(s,x)$,
\item\label{cc} $\bbC(x,x)=\bbC(x,i-)\tensor\bbC(i-,x)$, or explicitly: $\bbC(x,x)=\bigvee_{s\in S}\bbC(x,s)\circ\bbC(s,x)$,
\item\label{dd} for every $F,G\:\bbC\to\bbD$, if $F_{|S}\cong G_{|S}$ then $Fx\cong Gx$.
\end{enumerate}
\end{proposition}
\begin{proof}
(\ref{aa} $\impl$ \ref{dd}) By density of $S$ in $\cl{S}$, whenever $F$ and $G$ agree (up to isomorphism) on $S$ then they necessarily do so on $\cl{S}$ too. In particular $Fx\cong Gx$ whenever $x\in\cl{S}$. 

(\ref{dd} $\impl$ \ref{cc}) For the functors 
$$F\:\bbC\to\P\bbC\:c\mapsto\bbC(-,i-)\tensor\bbC(i-,c)\quad\mbox{ and }\quad G=Y_{\bbC}\:\bbC\to\P\bbC\:c\mapsto\bbC(-,c)$$
we have (much as in the proof of Proposition \ref{22}) for any $s\in S$ that $Fs=\bbC(-,i-)\tensor\bbC(i-,s)=\bbC(-,s)=Gs$. So $F_{|S}\cong G_{|S}$, and therefore $F\cong G$ by assumption, from which $\bbC(x,x)=(Gx)(x)=(Fx)(x)=\bigvee_{s\in S}\bbC(x,s)\circ\bbC(s,x)$ follows.

(\ref{cc} $\impl$ \ref{bb}) Is trivial.

(\ref{bb} $\impl$ \ref{aa}) For $T=\{x\in\bbC_0\mid1_{tx}\leq\bigvee_{s\in S}\bbC(x,s)\circ\bbC(s,x)\}$
we surely have $S\subseteq T$; so let $j\:\bbS\hookto\bbT$ be the corresponding full embedding. For any $x,y\in T$ we have $\bbT(y,x)=\bbC(y,x)$, so we can use the composition inequality in $\bbT$ to compute that
\begin{eqnarray*}
\bbT(y,x)
 & \geq & \bigvee_{s\in S}\bbT(y,js)\circ\bbT(js,x) \\
 & \geq & \bigvee_{s\in S}\bbT(y,x)\circ\bbT(x,js)\circ\bbT(js,x) \\
 & \geq & \bbT(y,x)\circ\bigvee_{s\in S}\bbT(x,js)\circ\bbT(js,x) \\
 & \geq & \bbT(y,x)\circ 1_{tx} \\
 & = & \bbT(y,x)
\end{eqnarray*}
This shows $S$ to be fully dense in $T$, and therefore $T\subseteq\cl{S}$. 
\end{proof}
Next we prove that the term `closure' is well-chosen:
\begin{proposition}\label{27}
For every $\Q$-category $\bbC$, $(\bbC_0,\cl{(\cdot)})$ is a closure space, and for every functor $F\:\bbC\to\bbD$, $F\:(\bbC_0,\cl{(\cdot)})\to(\bbD_0,\cl{(\cdot)})$ is a continuous function. This makes for a functor $\Cat(\Q)\to\Clos$.
\end{proposition}
\begin{proof}
It is straightforward to check that $S\mapsto\cl{S}$ is a monotone and increasing operation on the subsets of $\bbC_0$. As $S$ is fully dense in $\cl{S}$, which itself is fully dense in $\cl{\cl{S}}$, and the composition of two fully dense functors is again fully dense, it follows easily that $S$ is fully dense in $\cl{\cl{S}}$, so $\cl{\cl{S}}\subseteq\cl{S}$. This makes $(\bbC_0,\cl{(\cdot)})$ a closure space.

Now fix $S\subseteq\bbC_0$, and suppose that $x\in\cl{S}$. Functoriality of $F\:\bbC\to\bbD$ implies that 
$$1_{tFx}=1_{tx}\leq\bigvee_{s\in S}\bbC(x,s)\circ\bbC(s,x)\leq\bigvee_{s\in S}\bbD(Fx,Fs)\circ\bbD(Fs,Fx)= \bigvee_{t\in FS}\bbD(Fx,t)\circ\bbD(t,Fx),$$
which goes to show that $Fx\in\cl{FS}$. This makes $F\:(\bbC_0,\cl{(\cdot)})\to(\bbD_0,\cl{(\cdot)})$ a continuous function.

The functoriality of these constructions is a mere triviality.
\end{proof}
The following example nicely relates to Subsection \ref{X}.
\begin{example}\label{31.1}
Via the Yoneda embedding $Y_{\bbC}\:\bbC\to\P\bbC\:x\mapsto\bbC(-,x)$ we may consdider any $\Q$-category $\bbC$ as a full subcategory of the presheaf $\Q$-category $\P\bbC$: so $Y_{\bbC}(\bbC)$ is precisely the full subcategory of representable presheaves. For any presheaf $\phi\:\1_X\dist\bbC$ we may compute -- using the Yoneda Lemma -- that
$$\phi\in\cl{Y_{\bbC}(\bbC)}\iff 1_X\leq\bigvee_{x\in\bbC_0}\P\bbC(\phi,Y_{\bbC}x)\circ\P\bbC(Y_{\bbC}x,\phi)\iff 1_X\leq\bigvee_{x\in\bbC_0}\P\bbC(\phi,Y_{\bbC}x)\circ\phi(x).$$
On the other hand, in $\Dist(\Q)$ we have (as in any quantaloid) that $\phi\:\1_X\dist\bbC$ is a left adjoint if and only if its lifting through the identity, namely $\phi\se\bbC:\bbC\dist\1_X$, is its right adjoint, if and only if
$$1_X\le\bigvee_{x\in\bbC_0}(\phi\se\bbC)(x)\circ \phi(x)$$
holds. Because $(\phi\se\bbC)(x)=\phi\se\bbC(-,x)=\P\bbC(\phi,Y_\bbC(x))$ we thus find that $\phi\in\cl{Y_{\bbC}(\bbC)}$ exactly when $\phi$ is a left adjoint; or in words: {\em the Cauchy completion $\bbC\cc$ of $\bbC$ is the categorical closure of $\bbC$ in the free completion $\P\bbC$}.
\end{example}

\subsection{Strong Cauchy bilaterality---revisited}

Suppose now that $\Q$ is an involutive quantaloid (and, as usual, write $f\mapsto f\o$ for the involution). When $\bbC$ is a $\Q$-category and $S\subseteq\bbC_0$ determines the full subcategory $\bbS\hookto\bbC$, then that same set $S$ also determines a full subcategory $\bbS\s\hookto\bbC\s$ of the symmetrisation $\bbC\s$ of $\bbC$. Thus we may compute {\it two} closures of $S$: for notational convenience, let us write $\cl{S}$ for its closure in $\bbC$, and $\scl{S}$ for its closure in $\bbC\s$. We can then spell out that, for any $x\in\bbC_0$,
\begin{equation}\label{25}
x\in\cl{S}\iff 1_{tx}\leq\bigvee_{s\in S}\bbC(x,s)\circ\bbC(s,x)
\end{equation}
whereas
\begin{equation}\label{26}
x\in\scl{S}\iff 1_{tx}\leq\bigvee_{s\in S}\bbC\s(x,s)\circ\bbC\s(s,x)\iff1_{tx}\leq\bigvee_{s\in S}(\bbC(x,s)\wedge\bbC(s,x)\o)\circ(\bbC(s,x)\wedge\bbC(x,s)\o).
\end{equation}
It is straightforward that the second condition implies the first (without any further condition on $\Q$), so that $\scl{S}\subseteq\cl{S}$. This inclusion can be strict---but we have that:
\begin{proposition}\label{28}
For an involutive quantaloid $\Q$, the following conditions are equivalent:
\begin{enumerate}
\item for every $\Q$-category $\bbC$ and every subset $S\subseteq\bbC_0$, the closure of $S$ in $\bbC$ coincides with the closure of $S$ in $\bbC\s$,
\item $\Q$ is strongly Cauchy bilateral: for every family $(f_i\:X\to Y_i, g_i\:Y_i\to X)_{i\in I}$  of morphisms in $\Q$, $1_X\leq\bigvee_ig_i\circ f_i$ implies $1_X\leq\bigvee_i(g_i\wedge f_i\o)\circ(g_i\o\wedge f_i)$. 
\end{enumerate}
\end{proposition}
\begin{proof}
We continue with the notations introduced before the statement of this Proposition. If we apply the second condition to the family
$$(\bbC(s,x)\:tx\to tx,\bbC(x,s)\:ts\to tx)_{s\in S}$$
then we obtain immediately that $x\in\scl{S}$ whenever $x\in\cl{S}$, so $\cl{S}=\scl{S}$.

Conversely, given the family of morphisms in the second condition, define the $\Q$-category $\bbC$ with object set $\bbC_0=I\uplus\{x\}$, types given by $tx=X$ and $ti=Y_i$, and homs given by
$$\bbC(i,x)=f_i,\ \bbC(x,i)=g_i,\ \bbC(x,x)=1_X,\mbox { and }\bbC(j,i)=\left\{\begin{array}{cl}0_{Y_j,Y_i} & \mbox{ when }i\neq j \\ 1_{Y_i} & \mbox{ when }i=j\end{array}\right..$$
By assumption we must have $\cl{I}=\scl{I}$ for the subset $I\subseteq\bbC_0$, so in particular $x\in\cl{I}$ must imply $x\in\scl{I}$. Spelling this out with the aid of Equations \eqref{25} and \eqref{26} reveals the required formulas.
\end{proof}
In \cite{heystu11}, the notion of a `strongly Cauchy bilateral' quantaloid $\Q$ was introduced as a purely formal stonger version of (``ordinary'') Cauchy bilaterality, because in several examples the stronger version holds, and it is easier to verify. Here now, in the context of closures on $\Q$-categories, we have an explanation for the strong Cauchy bilaterality of $\Q$ as encoding precisely that ``closures can be symmetrised''. (But we do repeat that, for an integral quantaloid, strong Cauchy bilaterality and (`ordinary') Cauchy bilaterality are equivalent.) Whereas the Cauchy bilaterality of an involutive quantaloid $\Q$ implies that there is a distributive law of the Cauchy monad over the symmetrisation comonad on $\Cat(\Q)$ \cite[Corollary 3.9]{heystu11}, we can express strong Cauchy bilaterality of $\Q$ to mean that the functor $\Cat(\Q)\to\Clos$ is invariant under composition with the symmetrisation comonad $(-)\s\:\Cat(\Q)\to\Cat(\Q)$.

\subsection{Groundedness and additivity}\label{gradd}

The final issue we wish to address here in full generality, concerns the topologicity of the closure associated with any $\Q$-category---and this turns out to be a rather subtle point. Recall that a closure is said to be {\it topological} when it is both {\it grounded} (i.e.\ $\cl{\emptyset}=\emptyset$) and {\it additive} (i.e.\ $\cl{S\cup T}=\cl{S}\cup\cl{T}$). Especially when considering (convergence of) sequences in a closure space -- as we shall wish to do in the next section in the case of partial metric spaces -- it is problematic if that closure is non-grounded: for then any sequence converges to every point in $\cl{\emptyset}$.

First, for any $\Q$-category $\bbC$ it is easy to check that $\cl{\emptyset}=\{x\in\bbC_0\mid 1_{tx}=0_{tx}\}$; but for an object $Z$ in $\Q$ we have that $0_Z=1_Z$ if and only if $Z$ is a zero object (both terminal and initial); therefore $\cl{\emptyset}=\emptyset$ if and only $\bbC_0$ has no element whose type is a zero object in $\Q$. Conversely, if $\Q$ has a zero object $Z$, then quite obviously the categorical closure of the $\Q$-category $\1_Z$ does not satisfy $\cl{\emptyset}=\emptyset$. That is to say, the functor $\Cat(\Q)\to\Clos$ of Proposition \ref{27} factors through the full subcategory $\Clos\gd$ of grounded closure spaces if and only if $\Q$ does not have a zero object.
\begin{example}
Any non-trivial quantale -- viewed as a one-object quantaloid -- does not have a zero object, and therefore the categorical closure on such a quantale-enriched category is always grounded. However, every quantaloid of diagonals (our main concern in this paper) has zero objects: indeed, every zero morphism in a quantaloid $\Q$ determines a zero object in $\D(\Q)$. In particular, even when $\Q$ is a non-trivial quantale, $\D(\Q)$ will still have exactly one zero object. The categorical closure on a $\D(\Q)$-enriched category may thus very well be ungrounded---and thus we must be a little bit more careful when studying (convergence of) sequences in such an enriched category. The case that springs to mind is Lawvere's quantale of positive reals, $R=([0,\infty]\op,+,0)$, where $R$-categories (generalised metric spaces, cf.\ Example \ref{d}) have a grounded closure, but $\D(R)$-categories (generalised partial metric spaces, cf. Example \ref{9}) may have an ungrounded closure.
\end{example}
%
%\begin{example}
%The Lawvere quantale of positive reals, $R=([0,\infty]\op,+,0)$, does not have a zero object; as a consequence, the categorical closure on an $R$-category (a generalised metric space, cf.\ Example \ref{d}) is grounded. On the contrary, the quantaloid $\D(R)$ of diagonals in $R$ does have a zero object, namely $\infty$. (Here we have an unfortunate clash of categorical terminology and numeral notation.) Therefore, the categorical closure on a $\D(R)$-category (a generalised partial metric space, cf. Example \ref{9}) is not grounded, and we must be careful when studying sequences in a partial metric space!
%\end{example}
%
However, if $\Q$ does have a (unique\footnote{A similar reasoning holds when $\Q$ has several (necessarily uniquely isomorphic) zero objects, but we shall not need encounter that situation further on; indeed, our main concern is $\Q=\D(R)$.}) zero object $Z$, we can always ``discard'' the elements of type $Z$ from any given $\Q$-category $\bbC$: more precisely, if we define its full subcategories $\bbC\z$ and $\bbC\nz$ to have as elements
$$(\bbC\z)_0=\{x\in\bbC_0\mid tx=Z\}\quad\mbox{ and }\quad(\bbC\nz)_0=\{x\in\bbC_0\mid tx\neq Z\}$$
then $\bbC$ is exactly their categorical sum (coproduct): 
$$\bbC=\bbC\nz+\bbC\z.$$
Because any $\Q$-functor $F\:\bbC\to\bbD$ preserves types, it restricts to elements of non-zero type as $F\nz\:\bbC\nz\to\bbD\nz$. It follows easily that the canonical injection $i\:\bbC\nz\to\bbC$ is the counit for a (strictly) idempotent comonad on $\Cat(\Q)$, whose category of coalgebras $\Cat(\Q)\nz$ is exactly the full coreflective subcategory of those $\Q$-categories that do not have elements of type $Z$:
$$\Cat(\Q)\nz\xymatrix@=8ex{\ar@{^{(}->}@<-0.5ex>@/_1.5ex/[r]\ar@{}[r]|{\top} & \ar@{.>}@<-0.5ex>@/_1.5ex/[l]_{(-)\nz}}\Cat(\Q)$$
Furthermore, if we write $\Q\nz$ for the (smaller) quantaloid obtained from $\Q$ by discarding its zero object $Z$, then $\Cat(\Q)\nz=\Cat(\Q\nz)$ (and the full embedding $\Cat(\Q)\nz=\Cat(\Q\nz)\hookrightarrow\Cat(\Q)$ is actually the change of base determined by the homomorphism $\Q\nz\hookrightarrow\Q$). 
This goes to show that we always have a factorisation
$$\xymatrix@C=2ex@R=8ex{
\Cat(\Q)\nz\ar@{=}[r]&\Cat(\Q\nz)\ar@{.>}[d]\ar@{^{(}->}[rrr] &&& \Cat(\Q)\ar[d] \\
&\Clos\gd\ar@{^{(}->}[rrr] &&& \Clos}$$
The study of (convergence of) sequences of elements in a $\Q$-category $\bbC$  (for the categorical closure) is most useful, not in the whole of $\bbC$, but in its ``non-zero coreflection'' $\bbC\nz$.

In Section \ref{topparmet} we shall consider (convergence and Cauchyness of) sequences in (the non-zero coreflection of) a generalised partial metric space, and we shall want to relate it to the categorical Cauchy completion. To prepare the ground, we make here a few general observations regarding the Cauchy completion of a $\Q$-category $\bbC$ in case the quantaloid $\Q$ has a (unique) zero object $Z$. For any $\Q$-category $\bbC$ there is a unique Cauchy distributor from $\1_Z$ to $\bbC$, namely $\phi\:\1_Z\dist\bbC$ with, for all $x\in\bbC_0$, the $\phi(x)\:Z\to tx$ being the unique element of $\Q(tx,Z)$. In other words, the $\Q$-category $\bbC\cc$ contains exactly one element of type $Z$, which means that 
$$\bbC\cc\cong(\bbC\cc)\nz+\1_Z.$$
On the other hand, a Cauchy presheaf on $\bbC\nz$ {\it as $\Q\nz$-category} is exactly a Cauchy presheaf on $\bbC\nz$ {\it as $\Q$-category whose type is not zero}. That is to say, the following square commutes:
$$\xymatrix@C=2ex@R=8ex{
\Cat(\Q)\ar[rrr]^{(-)\nz}\ar[d]_{(-)\cc} &&& \Cat(\Q\nz)\ar@{=}[r]\ar[d]^{(-)\cc} & \Cat(\Q)\nz \\
\Cat(\Q)\ar[rrr]_{(-)\nz} &&& \Cat(\Q\nz)\ar@{=}[r] & \Cat(\Q)\nz}$$
(where on the right hand side we do the Cauchy completion qua $\Q\nz$-enriched category!). As a consequence, we find:
\begin{proposition}\label{100}
For $\Q$ a quantaloid with a unique zero object $Z$ and $\bbC$ any $\Q$-category, we have that
$$\bbC\cc\cong(\bbC\nz)\cc+\1_Z\mbox{ in }\Cat(\Q)$$
where $\bbC\cc$ is the $\Q$-enriched Cauchy completion of $\bbC$ and $(\bbC\nz)\cc$ is the $\Q\nz$-enriched Cauchy completion of $\bbC\nz$ (whose elements are in fact the $\Q$-enriched Cauchy presheaves on $\bbC\nz$ whose type is not $Z$).
\end{proposition}
% (Actually, because the full inclusion $i\:\bbC\nz\hookrightarrow\bbC$ is dense, it is true that $\bbC\nz$ and $\bbC$ are Morita-equivalent $\Q$-categories, and therefore have isomorphic Cauchy completions {\it as $\Q$-categories}. But the point here is, that we may discard the elements of type zero in $\bbC$, and compute the ``interesting part'' of $\bbC\cc$ by computing the Cauchy completion of $\bbC\nz$ {\it as $\Q\nz$-category.)

Finally, we end with a comment on the additivity of the categorical closure on $\bbC$. As for any closure, it is always true that $\cl{S}\cup\cl{T}\subseteq\cl{S\cup T}$ for any $S,T\subseteq\bbC_0$, but this inclusion need not be an equality. Indeed, for an $x\in\bbC_0$ we have that
\begin{equation}\label{31}
\begin{array}{rcl}
x\in\cl{S\cup T}& \iff & 1_{tx}\leq\displaystyle\bigvee_{r\in S\cup T}\bbC(x,r)\circ\bbC(r,x)\\
 & \iff & 1_{tx}\leq(\displaystyle\bigvee_{s\in S}\bbC(x,s)\circ\bbC(s,x))\vee(\displaystyle\bigvee_{t\in T}\bbC(x,t)\circ\bbC(t,x)),
 \end{array}
\end{equation}
whereas
\begin{equation}\label{32}
x\in\cl{S}\cup\cl{T}\quad\iff\quad 1_{tx}\leq\bigvee_{s\in S}\bbC(x,s)\circ\bbC(s,x)\quad\mbox{ or }\quad 1_{tx}\leq\bigvee_{t\in T}\bbC(x,t)\circ\bbC(t,x).
\end{equation}
It is now straightforward to identify a sufficient condition for the closure of any $\Q$-category to be topological (i.e.\ grounded and additive), which turns out to be also necessary when $\Q$ is integral). Admittedly this is not the most elegant condition---but it serves our purposes in the upcoming subsections.
\begin{proposition}\label{29}
For any quantaloid $\Q$, if every identity arrow is finitely join-irreducible\footnote{We mean here that, for any object $X$ of $\Q$, if $1_X\leq f_1\vee...\vee f_n$ ($n\in\mathbb{N}$) then $1_X\leq f_i$ for some $i\in\{1,...,n\}$. In other words, $1_X\neq 0_X$ and for any $1_X\leq f\vee g$ we have $1_X\leq f$ or $1_X\leq g$.} then the closure associated to any $\Q$-category $\bbC$ is topological. For any integral quantaloid $\Q$ the converse holds too. 
\end{proposition}
\begin{proof}
For any $\Q$-category $\bbC$ it is easy to check that $\cl{\emptyset}=\{x\in\bbC_0\mid 1_{tx}=0_{tx}\}$; therefore $\cl{\emptyset}=\emptyset$ if and only if none of the identities in $\Q$ is a bottom element. It is furthermore clear from the comparison of \eqref{31} and \eqref{32} that finite join-irreducibility of identities in (any) $\Q$ suffices for closures to be topological. Conversely, and under the extra assumption that $\Q$ is integral, for any $f,g\in\Q(X,X)$ there is a $\Q$-category $\bbC$ with three objects of type $X$, say $x,y,z$, and hom-arrows
$$\bbC(x,y)=f,\ \bbC(y,z)=g,\ \bbC(x,z)=f\circ g,\mbox{ and all others are $1_X$}.$$
It is easy to compute with the formula in Proposition \ref{24}--\ref{bb} that 
$$y\in\cl{\{x,z\}}\iff 1_X\leq f\vee g,\quad y\in\cl{\{x\}}\iff 1_X\leq f\quad\mbox{ and }\quad y\in\cl{\{z\}}\iff 1_X\leq g.$$
Thus, if this closure is topological then $1_X$ must be finitely join-irreducible. 
\end{proof}
If a quantaloid $\Q$ has a (unique) zero object, then it can never satisfy the condition in the above proposition; but removing that zero object from $\Q$ may very well produce a quantaloid $\Q\nz$ that does satisfy the condition above.

\section{Topology from partial metrics}\label{topparmet}

\subsection{Finitely typed partial metric spaces}

From now on we shall apply the previous material to the particular case where the base quantaloid is the quantaloid of diagonals in the (divisible, commutative) Lawvere quantale $R=([0,\infty]\op,+,0)$. As before, we shall write an $R$-enriched category as $(X,d)$, and a $\D(R)$-enriched category as $(X,p)$, to insist on their understanding as generalised (partial) metric spaces---even though we shall of course use the fully general theory of quantaloid-enriched categories where we see fit.

The quantale $R$ has its identity finitely join-irreducible (because the order is linear). The quantaloid $\D(R)$ has a unique zero object, namely $\infty$ (an unfortunate notational clash, due to the reversal of the natural order on $[0,\infty]$), but once we remove this zero object, the resulting quantaloid $\D(R)\nz$ has all its identities finitely join-irreducible (because the local order is linear). We saw in Examples \ref{12} and \ref{13} that $R$ and $\D(R)$ are both strongly Cauchy bilateral; {\it a fortiori} the same is true for the subquantaloid $\D(R)\nz$. 

The categorical closure on a generalised partial metric space $(X,p)$ is, as indicated in the previous section, non-grounded as soon as there exists an $x\in X$ such that $p(x,x)=\infty$. Excluding the points of self-distance\footnote{But we insist that for $x\neq y$ in $X$ it may still happen that $p(x,y)=\infty$.} $\infty$ from $(X,p)$ (that is, those elements which are of type $\infty$ in $(X,p)$ qua $\D(R)$-enriched category), we make sure that the categorical closure on that {\it finitely typed part of $(X,p)$} is topological. 

Restricting our attention now to {\it finitely typed} generalised partial metric spaces -- by which we mean of course those partial metrics such that $p(x,x)<\infty$, so that in effect we consider categories enriched in $\D(R)\nz$ -- we may infer from Propositions \ref{28} and \ref{29} that:
\begin{proposition}\label{30}
The categorical closure on a finitely typed generalised partial metric space $(X,p)$ is topological, and is identical to the closure on the associated symmetric finitely typed generalised partial metric space $(X,p\s)$ (where $p\s(y,x)=p(y,x)\vee p(x,y)$).
\end{proposition}
Now, for a finitely typed generalised partial metric space $(X,p)$, we find from Proposition \ref{24} that, for any subset $S\subseteq X$ and any $x\in X$,
\begin{equation}\label{30.1}
\begin{array}{rcl}
x\in\cl{S}
 & \iff & p(x,x)\geq\displaystyle\bigwedge_{s\in S}p(x,s)-p(s,s)+p(s,x) \\
 & \iff & 0\geq\displaystyle\bigwedge_{s\in S}p(x,s)-p(s,s)+p(s,x)-p(x,x)
\end{array}
\end{equation}
The expression under the infimum is thus precisely equal to $p_0(x,s)+p_0(x,s)$ for the generalised metric $p_0$ associated with the partial metric $p$ through the change of base $J_0\:\D(R)\to R$ (see below Example \ref{16}). That is to say:
\begin{proposition}\label{33}
The categorical topology on a finitely typed generalised partial metric space $(X,p)$ is identical to the topology on the generalised metric space $(X,p_0)$ (where $p_0(y,x):=p(y,x)-p(x,x)$).
\end{proposition}
Putting both previous Propositions together, we can conclude that the categorical topology on a finitely typed generalised partial metric space {\it is always metrisable by means of a symmetric generalised metric}. And for such a {\it symmetric} generalised metric space $(X,d)$, Proposition \ref{24} says that
$$\begin{array}{rcl}
x\in\cl{S}
 & \iff & 0\geq\displaystyle\bigwedge_{s\in S} d(x,s)+d(s,x) \\
 & \iff & 0\geq\displaystyle\bigwedge_{s\in S} 2\cdot d(x,s) \\
 & \iff & 0\geq\displaystyle\bigwedge_{s\in S} d(x,s) \\
 & \iff & \forall\varepsilon>0\ \exists s\in S:d(x,s)<\varepsilon 
\end{array}$$
Thus the categorical topology on $(X,d)$ is exactly the usual metric topology---with a basis given by the collection of open balls
$$\Big\{B(x,\varepsilon):=\{y\in X\mid d(x,y)<\varepsilon\}\ \rule[-6pt]{0.75pt}{18pt}\ x\in X,\varepsilon>0\Big\},$$
with its usual notion of convergent sequences, etc. 

One could consider this a disappointment: there are not more ``partially metrisable topologies'' then there are metrisable ones. Still, one must realise that it is not always trivial to interpret topological and/or metric phenomena in a given finitely typed partial metric $(X,p)$ by passing to some metric $(X,d)$ which just happens to define the same topology. The next subsection is entirely devoted to the study of convergent sequences in finitely typed partial metrics.

\subsection{Convergence \dots}

One fundamental use of topology is its inherent notion of convergence for sequences: $(x_n)_n\to x$ in a topological space $(X,\T)$ when for every $x\in U\in \T$ there exists an $n_0$ such that $x_n\in U$ for every $n\geq n_0$. When the topology stems from a symmetric generalised metric $d$ on $X$, it is sufficient to consider open balls centered in $x$, and so 
$$(x_n)_n\to x\iff\forall\varepsilon>0\ \exists n_0\ \forall n\geq n_0:d(x_n,x)<\varepsilon.$$
A convergent sequence is necessarily a Cauchy sequence, meaning that
$$\forall\varepsilon>0\ \exists n_0\ \forall m,n\geq n_0:d(x_n,x_m)<\varepsilon,$$
and a symmetric generalised metric space is said to be (sequentially) Cauchy complete precisely when every Cauchy sequence converges. But note that the definition of Cauchy sequence is symmetric in $x_n$ and $x_m$ {\it even when the generalised metric $d$ is not symmetric}---and so it makes perfect sense for {\it any} generalised metric space. As recalled in Example \ref{l}, Lawvere \cite{law72} proved that a generalised metric space $(X,d)$ is sequentially Cauchy complete if and only if every left adjoint distributor into $(X,d)$ (now viewed as an $R$-enriched category) is representable.

Now consider a finitely typed generalised partial metric space $(X,p)$; its categorical topology is equivalently described by the symmetric generalised metric 
\begin{equation}\label{34}
(p_0)\s(y,x)=p_0(y,x)\vee p_0(x,y)=(p(y,x)-p(x,x))\vee(p(x,y)-p(y,y)),
\end{equation}
and therefore a sequence $(x_n)_n$ in $(X,p)$ converges to $x\in X$ precisely when
\begin{equation}\label{34.0}
\forall\varepsilon>0\ \exists n_0\ \forall n\geq n_0:(p(x,x_n)-p(x_n,x_n))\vee(p(x_n,x)-p(x,x))<\varepsilon.
\end{equation}
In what follows we shall first try to improve on our understanding of this formula, and thus of convergence in $(X,p)$, before we look at Cauchy sequences and completion.

For any quantaloid $\Q$, the terminal object $\bbT$ in $\Cat(\Q)$ (exists and) has the following description: its object set is $\bbT_0=\Q_0$, the type function is the identity, and the hom function is $\bbT_0\times\bbT_0\to\Q_1\:(Y,X)\mapsto\top_{X,Y}$ (where $\top_{X,Y}$ is the top element in $\Q(X,Y)$). The unique functor from a $\Q$-category $\bbC$ to $\bbT$ is $\bbC_0\to\bbT_0\:x\mapsto tx$, that is, it is $\bbC$'s type function. From Proposition \ref{27} we deduce that the type function of a $\Q$-category is continuous---but, of course, the use of this statement depends on the categorical topology of $\bbT$. If we work over the quantaloid $\D(R)\nz$ (so that $\Cat(\D(R)\nz)=\Cat(\D(R))\nz$ is exactly the category of {\it finitely typed} partial metric spaces), then things are as follows:
\begin{proposition}\label{34.1}
The terminal finitely typed generalised partial metric space $(T,p)$ is defined by $T=[0,\infty[$ and $p(a,b)=a\vee b$; its categorical topology is the usual metric topology.
\end{proposition}
\begin{proof}
The general construction of the terminal $\D(R)\nz$-category $\bbT$ -- which we henceforth write as a generalised partial metric space $(T,p)$ -- says that 
$$T=\mbox{objects of }\D(R)\nz=[0,\infty[\mbox{ and }p(a,b)=\mbox{top element of }\D(R)\nz(a,b)=a\vee b.$$
From the above discussion, the categorical topology on the partial generalised metric $(T,p)$ is equivalently described by the (``total'') generalised metric $(T,p_0)$, which in turn is equivalently described by its symmetrisation $(X,(p_0)\s)$. A simple computation leads to
$$(p_0)\s(a,b)=p_0(a,b)\vee p_0(b,a)=((a\vee b)-a)\vee(a\vee b)-b)=|a-b|.$$
That is to say, the categorical topology on $(T,p)$ (qua partial metric space) is precisely the usual metric topology.
\end{proof}
\begin{corollary}\label{35}
For any finitely typed generalised partial metric space $(X,p)$, equipped with its categorical topology,
\begin{enumerate}
\item the function $X\to[0,\infty[\:x\mapsto p(x,x)$ is continuous (for the usual topology on $[0,\infty[$),
\item if $(x_n)_n\to x$ in $(X,p)$ then $\lim_{n\to\infty}p(x_n,x_n)=p(x,x)$.
\end{enumerate}
\end{corollary}
We can now prove a practical characterisation of convergence in a finitely typed partial metric space, which subsumes the definition of convergence from \cite{bukatinetal09} (in the case $p$ is symmetric, separated and never takes the value $\infty$) and improves the one given in \cite{puzha12} (in that it eliminates all double limits\footnote{Precisely, in \cite{puzha12} it is required that $\lim_{m,n\to\infty}p(x_m,x_n)=p(x,x)$ instead of our $\lim_{n\to\infty}p(x_n,x_n)=p(x,x)$. Conceptually, our simple limit expresses that the type of $x_n$ should converge to the type of $x$ (but nothing more); therefore our notion of convergence is directly applicable to the typed sequences of Definition \ref{37.0} further on.}):
\begin{proposition}\label{36}
In a finitely typed generalised partial metric space $(X,p)$, equipped with its categorical topology, we have a convergent sequence $(x_n)_n\to x$ if and only if all three limits 
$$\lim_{n\to\infty}p(x,x_n),\ \lim_{n\to\infty}p(x_n,x_n)\mbox{ and }\lim_{n\to\infty}p(x_n,x)$$
(exist and) are equal to $p(x,x)$.
\end{proposition}
\begin{proof}
Suppose first that $(x_n)_n\to x$ in $(X,p)$. Because $p(x,x_n)\wedge p(x_n,x)\geq p(x,x)\vee p(x_n,x_n)$, the expression in \eqref{34.0} is equivalent to
$$\forall\varepsilon>0\ \exists n_0\ \forall n\geq n_0:\ 
\left\{\begin{array}{l}
|p(x,x_n)-p(x_n,x_n)|<\varepsilon \\
|p(x_n,x)-p(x,x)|<\varepsilon
\end{array}\right.$$
and so in particular $\lim_{n\to\infty}p(x_n,x)=p(x,x)$. Corollary \ref{35} assures that $\lim_{n\to\infty}p(x_n,x_n)=p(x,x)$, that is,
$$\forall\varepsilon>0\ \exists n_1\ \forall n\geq n_1:\ |p(x_n,x_n)-p(x,x)|<\varepsilon,$$
and so for any $n\geq n_0\vee n_1$ also
$$|p(x,x_n)-p(x,x)|\leq|p(x,x_n)-p(x_n,x_n)|+|p(x_n,x_n)-p(x,x)|<2\varepsilon.$$
Therefore $\lim_{n\to\infty}p(x,x_n)=p(x,x)$ too. Hence we proved the necessity of the three limits.

Conversely, knowing that $\lim_{n\to\infty}p(x,x_n)=p(x,x)=\lim_{n\to\infty}p(x_n,x_n)$, we find also
$$0\leq|p(x,x_n)-p(x_n,x_n)|\leq|p(x,x_n)-p(x,x)|+|p(x,x)-p(x_n,x_n)|,$$
whence $\lim_{n\to\infty}(p(x,x_n)-p(x_n,x_n))=0$. Together with $\lim_{n\to\infty}p(x_n,x)=p(x,x)$ this shows the sufficiency of the three limits.
\end{proof}
The middle limit in the above proposition is crucial, as the following example indicates:
\begin{example}\label{36.1}
%%% This example is not finitely typed:
%In the terminal partial metric space from Proposition \ref{34.1}, any sequence $(x_n)_n$ satisfies $\lim_{n%\to\infty}p(x_n,\infty)=\infty=\lim_{n\to\infty}p(\infty,x_n)$, but it would be against all intuition to %have every sequence $(x_n)_n$ converging to $\infty$! Requiring that $\lim_{n\to\infty}p(x_n,x_n)=\infty$ %too, precisely excludes such pathological behaviour.
%%%
For $A$ a (non-empty) finite alphabet, let $X$ be the union of all non-empty words and all sequences in that alphabet: it is a finitely typed generalised partial metric space if we put $p(x,y)=(\frac{1}{2})^k$ where $k$ is the position of the first letter in which $x$ and $y$ do not agree. Now consider a sequence $(x_n)_n$ with $x_0\in A$ and each $x_{n+1}$ is equal to $x_n$ concatenated with one extra letter: we then have that $\lim_{n\to\infty}p(x_0,x_n)=p(x_0,x_0)=\lim_{n\to\infty}p(x_n,x_0)$, but it is against all intuition to say that $(x_n)_n$ converges to $x_0$! Precisely because $\lim_{n\to\infty}p(x_n,x_n)\neq p(x_0,x_0)$ this pathological behaviour is excluded.
\end{example}
Note that Proposition \ref{36} contains the usual convergence criterion in an ordinary metric space, where we would have $p(x,x)=0=p(x_n,x_n)$ and $p(x,x_n)=p(x_n,x)$. 

\subsection{\dots\ and completeness}

We now turn to the study of Cauchy sequences in, and completion of, finitely typed partial metric spaces (for the categorical topology). 

Recall that a finitely typed generalised partial metric space $(X,p)$ is a $\D(R)\nz$-category $\bbX$ with object set $\bbX_0=X$, type function $tx=p(x,x)$ and hom-arrows $\bbX(y,x)=p(y,x)$. The crucial rôle of the type function as ``indicator of partialness'' was already apparent in the previous subsection. To facilitate our discussion of sequences in $(X,p)$ we find it useful to introduce some further terminology:
\begin{definition}\label{37.0}
A sequence $(x_n)_n$ in a finitely typed generalised partial metric space $(X,p)$ is \textbf{typed} whenever $\lim_{n\to\infty}p(x_n,x_n)$ exists in $[0,\infty[$; that limit is then called the type of $(x_n)_n$.
\end{definition}
Because we only consider sequences in a {\it finitely typed} $(X,p)$ (for the reasons explained in Subsection \ref{gradd}), any typed sequence is in fact of finite type too.
\begin{lemma}\label{37.1}
For any finitely typed generalised partial metric space $(X,p)$, the following defines an equivalence relation on the set of all typed sequences in $(X,p)$:
$$(x_n)_n\sim (y_n)_n \ \stackrel{\rm def.}{\Longleftrightarrow} \ \lim_{n\to\infty}p(x_n,y_n)=\lim_{n\to\infty}p(x_n,x_n)=\lim_{n\to\infty}p(y_n,y_n)=\lim_{n\to\infty}p(y_n,x_n).$$
\end{lemma}
\begin{proof}
Reflexivity and symmetry are obvious. If $(x_n)_n\sim(y_n)_n$ and $(y_n)_n\sim(z_n)_n$ then all three must have the same type, say $q$, and since both extremes of the double inequality 
$$p(x_n,x_n)\vee p(z_n,z_n)\leq p(x_n,z_n)\leq p(x_n,y_n)-p(y_n,y_n)+p(y_n,z_n)$$ 
converge to $q$ when $n$ goes to $\infty$, so does the middle term. Similarly, $\lim_{n\to\infty}p(z_n,x_n)=q$.
\end{proof}
Let us stress that the equivalence relation only pertains to typed sequences, and that equivalent sequences necessarily have the same type. As was also done in \cite{mat94}, we furthermore define:
\begin{definition}\label{37}
A sequence $(x_n)_n$ in a finitely typed generalised partial metric space $(X,p)$ is {\bf Cauchy} if $(p(x_n,x_m))_{(n,m)}$ is a Cauchy net in $[0,\infty]$.
\end{definition}
Here we regard $[0,\infty]$ canonically as a generalised metric space: $d(a,b)=\max(a-b,0)$. By our general considerations, its categorical topology is metrisable by the symmetric distance function
$$d_s(a,b)=d(a,b)\vee d(b,a)=\left\{\begin{array}{l} 0\mbox{ if }a=\infty=b \\ |a-b|\mbox{ if }a\neq\infty\neq b \\ \infty\mbox{ otherwise}\end{array}\right.$$
If a net $(a_{(m,n)})_{(m,n)}$ is Cauchy in $[0,\infty]$ then it is either eventually constant $\infty$ or eventually finite. Since $(X,p)$ is finitely typed by assumption, the former cannot happen for $a_{(m,n)}=p(x_n,x_m)$, so every such Cauchy net lies eventually\footnote{Thus this notion of Cauchyness narrows down to the one in \cite{bukatinetal09,mat94} which is only concerned with partial metrics satisfying $p(x,y)<\infty$ for all $x,y\in X$.} in $[0,\infty[$. As this is a complete space, this implies that the Cauchy net $(p(x_n,x_m))_{(n,m)}$ converges to the ``double'' limit $\lim_{m,n\to\infty}p(x_n,x_m)$ in the usual sense (see \cite{kel55} for details on sequences and nets). 

The following results use the equivalence relation on typed sequences to express the expected interplay between convergent sequences and Cauchy sequences in partial metric spaces:
\begin{proposition}\label{38}
In a finitely typed generalised partial metric space $(X,p)$,
\begin{enumerate}
\item any constant sequence $(x)_n$ is typed, with type $p(x,x)$,
\item $(x_n)_n\to x$ if and only if $(x_n)_n$ (is typed and) $(x_n)_n\sim(x)_n$,
\item if $(x_n)_n\to x$ and $(y_n)_n$ is typed, then $(x_n)_n\sim (y_n)_n$ if and only if $(y_n)_n\to x$,
\item any Cauchy sequence $(x_n)_n$ is typed, with type $\lim_{m,n\to\infty}p(x_n,x_m)$,
\item if $(x_n)_n\sim(y_n)_n$ then either one is Cauchy if and only if the other one is too,
\item every convergent sequence is Cauchy.
\end{enumerate}
\end{proposition}
\begin{proof}
(1) Is trivial.
(2) This is a reformulation of Proposition \ref{36}. 
(3) Follows from the previous assertion and transitivity of $\sim$.
(4) If the net $(p(x_n,x_m))_{(n,m)}$ converges in $[0,\infty]$, then the subnet $(p(x_n,x_n))_n$ converges to the same value. 
(5) Assume that $(x_n)_n$ and $(y_n)_n$ have type $q$ and that $(y_n)_n$ is a Cauchy sequence. Then, for all natural numbers $n$ and $m$,
$$p(x_n,x_n)\leq p(x_n,x_m)\leq p(x_n,y_n)-p(y_n,y_n)+p(y_n,y_m)-p(y_m,y_m)+p(y_m,x_m);$$
and since both extremes of this double inequality converge to $q$ when $(n,m)$ goes to $(\infty,\infty)$, so does $p(x_n,x_m)$.
(6) If $(x_n)_n\to x$, then $(x_n)_n\sim (x)_n$; since $(x)_n$ is Cauchy, $(x_n)_n$ is so too.
\end{proof}

To convince the categorically inclined that Definition \ref{37} makes perfect sense, we want to show that there is an essentially bijective correspondence between Cauchy sequences in a finitely typed partial metric space $(X,p)$ on the one hand, and Cauchy distributors on the $\D(R)\nz$-category $\bbX$ (still defined by $\bbX_0=X$, $tx=p(x,x)$ and $\bbX(y,x)=p(y,x)$, of course). Recall that a $\D(R)\nz$-distributor $\phi:\1_q\dist\bbX$ is (in terms of the partial metric) defined by a number $q\in[0,\infty[$ together with a function $\phi\:X\to[0,\infty]$ such that 
\begin{equation}\label{39.1}
q\vee p(y,y)\leq\phi(y)\leq p(y,x)-p(x,x)+\phi(x)
\end{equation}
for all $x,y\in X$. Similarly, a $\D(R)\nz$-distributor $\psi\:\bbX\dist \1_q$ is a number $q\in[0,\infty[$ together with a function $\psi\:X\to[0,\infty]$ such that 
\begin{equation}\label{39.2}
q\vee p(y,y)\leq\psi(y)\leq\psi(x)-p(x,x)+p(x,y)
\end{equation}
for all $x,y\in X$. Such distributors form an adjoint pair $\phi\dashv\psi$ (and so $\phi$ is a Cauchy presheaf, and then we rather write $\phi^*=\psi$) if and only if $\1_q\leq\psi\otimes\phi$ and $\phi\otimes\psi\le \bbX$ in $\Dist(\D(R)\nz)$, that is, for all $x,y\in X$,
\begin{equation}\label{39.3}
\bigwedge_{z\in X}\psi(z)-p(z,z)+\phi(z)\le q \ \mbox{ and } \ p(y,x)\leq \phi(y)-q+\psi(x).
\end{equation}
Fixing $x\in X$, the representable presheaves $\bbX(-,x)\:\1_{tx}\dist\bbX$ and $\bbX(x,-)\:\bbX\dist\1_{tx}$ always form an adjoint pair; they correspond to the functions $p(-,x)\:X\to[0,\infty]$ and $p(x,-)\:X\to[0,\infty]$, with $q=p(x,x)$.
%\begin{remark}\label{40.-2}
%We hasten to point out that, with the same notations as above, there is -- for any partial metric space $(X,p)$ -- a unique distributor from $\1_{\infty}$ to $X$, and a unique distributor from $X$ to $\1_{\infty}$, namely $\phi\:\1_{\infty}\dist X$ and $\psi\:X\dist\1_{\infty}$ determined by 
%$$\phi(x)=\infty=\psi(x);$$
%moreover, $\phi\dashv\psi$ always holds. This adjunction is represented by any point $x\in X$ of type $tx=\infty$ (provided $X$ has such a point).
%\end{remark}

\begin{lemma}\label{40.-1}
If $(x_n)_n$ and $(y_n)_n$ are Cauchy sequences in a finitely typed generalised partial metric space $(X,p)$ then $(p(x_n,y_m))_{n,m}$ is a Cauchy net in $[0,\infty]$.
\end{lemma}
\begin{proof}
Let $\varepsilon>0$. Since $(x_n)_n$ and $(y_n)_n$ are Cauchy sequences in $(X,p)$, there is a natural number $N$ so that for all $n,m,n',m'\ge N$,
$$p(x_n,x_{n'})-p(x_{n'},x_{n'})\le\varepsilon, \quad p(y_{m'},y_m)-p(y_{m'},y_{m'})\le\varepsilon,$$
$$p(x_{n'},x_n)-p(x_n,x_n) \le\varepsilon, \quad p(y_m,y_{m'})-p(y_m,y_m)\le\varepsilon.$$
From these inequalities (and the triangular inequality for $p$) we obtain
\begin{eqnarray*}
p(x_n,y_m)-p(x_{n'},y_{m'})
& \leq & (p(x_n,x_{n'})-p(x_{n'},x_{n'})+p(x_{n'},y_m))-p(x_{n'},y_{m'}) \\
& \leq & \varepsilon+p(x_{n'},y_m)-p(x_{n'},y_{m'}) \\
& \leq & \varepsilon+(p(x_{n'},y_{m'})-p(y_{m'},y_{m'})+p(y_{m'},y_m))-p(x_{n'},y_{m'}) \\
& \leq & \varepsilon+p(x_{n'},y_{m'})+\varepsilon-p(x_{n'},y_{m'}) \\
& \leq & 2\varepsilon;
\end{eqnarray*}
similarly (and simultanuously) $p(x_{n'},y_{m'})-p(x_n,y_m)\leq 2\varepsilon$ too. This tells us that $|p(x_n,y_m)-p(x_{n'},y_{m'})|\le 2\varepsilon$, for all $n,m,n',m'\ge N$, which establishes Cauchyness of the net.
\end{proof}
In particular, if $(x_n)_n$ is a Cauchy sequence in a finitely typed generalised partial metric space $(X,p)$ then, for every $y\in X$, both $(p(y,x_n))_n$ and $(p(x_n,y))_n$ are Cauchy sequences in $[0,\infty]$, and therefore converge. This guarantees the existence of the limits in the statement of the next theorem.
\begin{theorem}\label{39}
Let $(X,p)$ be a finitely typed generalised partial metric space, and $\bbX$ the corresponding $\D(R)\nz$-category (with $\bbX_0=X$, $tx=p(x,x)$ and $\bbX(y,x)=p(x,x)$, as always). If $(x_n)_n$ is a Cauchy sequence in $(X,p)$, and we put $q=\lim_{n,m\to\infty}p(x_n,x_m)$, then
$$\phi:\1_q\dist\bbX\mbox{ with elements }\phi(y)=\lim_{n\to\infty}p(y,x_n)$$
is a Cauchy presheaf of finite type, whose right adjoint is
$$\psi:\bbX\dist\1_q\mbox{ with elements }\psi(y)=\lim_{n\to\infty}p(x_n,y).$$
This correspondence is bijective between equivalence classes of Cauchy sequences on the one hand, and Cauchy distributors of finite type on the other. Moreover, a Cauchy sequence converges (to $x\in X$) if and only if the corresponding Cauchy distributor is representable (by $x\in X$). 
\end{theorem}
\begin{proof}
First we verify \eqref{39.1} to make sure that $\phi\:\1_q\dist\bbX$ is a well-defined presheaf on $\bbX$. Because $p$ is a partial metric we certainly have $p(y,y)\vee p(x_n,x_n)\leq p(y,x_n)\leq p(y,x)-p(x,x)+p(x,x_n)$ for all $n$. Letting $n$ go to $\infty$, we therefore find that $p(y,y)\vee q\leq\phi(y)\leq p(y,x)-p(x,x)+\phi(x)$, as required. A similar reasoning holds to verify \eqref{39.2} for $\psi$. 

To show that $\phi\dashv\psi$, we have to verify \eqref{39.3}; applied to the case at hand, this means that
$$\bigwedge_{z\in X}\lim_{n\to\infty}p(x_n,z)-p(z,z)+\lim_{n\to\infty}p(z,x_n)\le q$$
$$\mbox{ and }\quad p(y,x)\leq \lim_{n\to\infty}p(y,x_n)-q+\lim_{n\to\infty}p(x_n,x)$$
for all $y,z\in Y$. To see the first inequality, let $\varepsilon>0$. Since $(x_n)_n$ is a Cauchy sequence of type $q$ in $(X,p)$, there is a natural number $N$ so that for any $n\geq N$
$$p(x_n,x_N)\le q+\varepsilon,\quad p(x_N,x_n)\le q+\varepsilon\quad\mbox{ and }\quad q-\varepsilon\leq p(x_N,x_N).$$
Therefore
$$p(x_n,x_N)-p(x_N,x_N)+p(x_N,x_n)\le q+3\varepsilon,$$
and the assertion follows by choosing $\varepsilon$ arbitrarily small and letting $n$ go to $\infty$. To show the second inequality, for $\varepsilon>0$ let $N$ be a natural number so that, for all $n,m\ge N$,
$$p(x_n,x_m)-p(x_n,x_n)\le\varepsilon\quad\mbox{ and }\quad q-\varepsilon\le p(x_m,x_m)$$
for all $n,m\ge N$. It then follows that
\begin{eqnarray*}
p(y,x)
& \leq & p(y,x_n)-p(x_n,x_n)+p(x_n,x_m)-p(x_m,x_m)+p(x_m,y) \\
& \leq & p(y,x_n)+\varepsilon-p(x_m,x_m)+p(x_m,y) \\
& \leq & p(y,x_n)+\varepsilon-(q-\varepsilon)+p(x_m,y) \\
& \leq & p(y,x_n)-q+p(x_m,y)+2\varepsilon.
\end{eqnarray*}
Choosing $\varepsilon$ arbitrary small and letting $n$ and $m$ go to $\infty$, this proves the point.

If $(y_n)_n$ is a Cauchy sequence with $(y_n)_n\sim(x_n)_n$, then in particular $\lim_{n\to\infty}p(y_n,y_n)=\lim_{n\to\infty}p(y_n,x_n)$. For any $z\in X$ we know that
$$p(z,x_n)\le p(z,y_n)-p(y_n,y_n)+p(y_n,x_n),$$
and therefore $\lim_{n\to\infty}p(z,x_n)\le\lim_{n\to\infty}p(z,y_n)$. A similar argument shows the reverse inequality, which proves that both sequences define the same Cauchy distributor. Conversely, if two Cauchy sequences $(x_n)_n$ and $(y_n)_n$ induce the same Cauchy distributor $\phi:\1_q\dist\bbX$, with right adjoint $\psi:\bbX\dist\1_q$, then they are of the same type $q$. Moreover, for every $\varepsilon>0$, there exist some natural number $N$ so that, for all $n\ge N$
$$q\le\phi(x_n)=\lim_{m\to\infty}p(x_n,x_m)\le q+\varepsilon 
\quad\mbox{ and }\quad
q\le\psi(y_n)=\lim_{m\to\infty}p(y_m,y_n)\le q+\varepsilon.$$
Hence, $p(x_n,x_n)\le p(x_n,y_n)\le \phi(x_n)-q+\psi(y_n)\le q+2\varepsilon$, for all $n\ge N$, which proves $(x_n)_n\sim(y_n)_n$.

Let now $\phi\dashv\psi$ with $\phi:\1_q\dist\bbX$ and $\psi:\bbX\dist\1_q$ (for some $q\in[0,\infty[$). Thanks to the first inequation in \eqref{39.3} we can pick, for every natural number $n$, an element $x_n\in X$ so that
$$\phi(x_n)-p(x_n,x_n)+\psi(x_n)\le q+\frac{1}{n}.$$
But \eqref{39.1} and \eqref{39.2} say that $p(x_n,x_n)\vee q\le\phi(x_n)\wedge\psi(x_n)$, so we find\begin{eqnarray*}
&& p(x_n,x_n)\leq q+\frac{1}{n}\quad\mbox{ and }\quad q\leq p(x_n,x_n)+\frac{1}{n},\\
&& p(x_n,x_n) \le \phi(x_n)\le p(x_n,x_n)+\frac{1}{n},\\
&& \mbox{and }\quad p(x_n,x_n)\le \psi(x_n)\le p(x_n,x_n)+\frac{1}{n},
\end{eqnarray*}
which implies that 
\begin{equation}\label{40.0}
q=\lim_{n\to\infty}p(x_n,x_n)=\lim_{n\to\infty}\phi(x_n)=\lim_{n\to\infty}\psi(x_n).
\end{equation}
By the second inequation in \eqref{39.3} we know that
$$p(x_n,x_n)\leq p(x_n,x_m)\le\phi(x_n)-q+\psi(x_m)$$
for all $n$ and $m$, so with \eqref{40.0} we obtain $\lim_{n,m\to\infty}p(x_n,x_m)=q$, and we proved $(x_n)_n$ to be a Cauchy sequence in $(X,p)$. Finally, this Cauchy sequence in turn determines the Cauchy presheaf it was constructed from: because from \eqref{39.1} and \eqref{39.2} we get
$$\phi(x)\leq p(x,x_n)-p(x_n,x_n)+\phi(x_n)\quad\mbox{ and }\quad\psi(x)\leq\psi(x_n)-p(x_n,x_n)+p(x_n,x)$$
for all $x\in X$ and all natural numbers $n$, and with \eqref{40.0} we find that
$$\phi\le\lim_{n\to\infty}p(-,x_n)\quad\mbox{ and }\quad\psi\le\lim_{n\to\infty}p(x_n,-)$$
too; and these inequalities are equalities because $\phi\dashv\psi$ and (as attested by the first part of this proof) $\lim_{n\to\infty}p(-,x_n)\dashv\lim_{n\to\infty}p(x_n,-)$.
\end{proof}
Combining the above Theorem \ref{39} with the remarks in Subsection \ref{gradd}, we arrive at the following conclusions.
\begin{corollary}\label{40}
A generalised partial metric space $(X,p)$ is categorically Cauchy complete (meaning that every Cauchy distributor on $(X,p)$ qua $\D(R)$-enriched category is representable) if and only if the finitely typed part of $(X,p)$ is sequentially Cauchy complete (meaning that every Cauchy sequence in $(X,p)$ converges) and $(X,p)$ has at least one point of type $\infty$.
\end{corollary}
Especially Proposition \ref{100} helps us with:
\begin{example}\label{40.1}
The Cauchy completion of a generalised partial metric space $(X,p)$, viewed as a $\D(R)$-category $\bbX$, is the sum in $\Cat(\D(R))$ of the Cauchy completion {\it qua $\D(R)\nz$-enriched category} of $\bbX\nz$ plus a singleton of type $\infty$:
$$\bbX\cc=(\bbX\nz)\cc+\1_{\infty}.$$
But $\bbX\nz$ is exactly the finitely typed part of $(X,p)$, and we know by Theorem \ref{39} that the finitely typed Cauchy presheaves on the finitely typed part of $(X,p)$ are in one-to-one correpondence with equivalence classes of Cauchy sequences. Therefore, the Cauchy completion of $(X,p)$ has as elements the equivalence classes of Cauchy sequences in the finitely typed part of $(X,p)$, plus an extra point which we shall denote by $\infty$, and comes with the partial metric defined by
$$p(\infty,[(x_n)_n])=p([(x_n)_n],\infty)=p(\infty,\infty)=\infty$$
and
$$p([(x_n)_n],[(y_n)_n)])\stackrel{(1)}{=}\bigwedge_{z\in X}\lim_{n\to\infty}p(x_n,z)-p(z,z)+\lim_{n\to\infty}p(z,y_n)\stackrel{(2)}{=}\lim_{n\to\infty}p(x_n,y_n).$$
Indeed, the first equality in the line above is exactly the formula for the hom-arrow in $\bbX\cc$ between the corresponding Cauchy distributors; the second equality can be proven as follows. Thanks to Lemma \ref{40.-1} we know that $(p(x_n,y_m)_{(n,m)})$ is a Cauchy net in $[0,\infty]$, so it converges, and therefore so does the subnet $(p(x_n,y_n)_n)$; so we may put $q=\lim_{n\to\infty}p(x_n,y_n)$. Since we always have
$$p(x_n,y_n)\leq p(x_n,z)-p(z,z)+p(z,y_n)$$
we can let $n$ go to $\infty$, and then take the infimum over $z$, to see that the ``$\geq$'' in the second equality always holds. For the ``$\leq$'', let $\varepsilon>0$. Since both $(p(x_n,x_m))_{(n,n)}$ and $(p(x_n,y_m))_{(n,m)}$ are Cauchy nets in $[0,\infty]$ (as, again, attested by Lemma \ref{40.-1}), there is some natural number $N$ so that, for all $n\ge N$,
$$p(x_n,n_N)-p(x_N,x_N)\le\varepsilon\quad\mbox{ and }\quad p(x_N,y_n)\le q+\varepsilon.$$
Therefore
$$\lim_{n\to\infty}p(x_n,x_N)-p(x_N,x_N)+\lim_{n\to\infty}p(x_N,y_n)\le q+2\varepsilon,$$
and the assertion follows.
\end{example}
The above results for partial metric spaces of course apply to metric spaces too---and {\it almost} produce the ``usual'' results. Indeed, a Cauchy sequence in a (generalised) metric space $(X,d)$ in the sense of Definition \ref{37} is exactly a Cauchy sequence in the usual sense; and it converges in $(X,d)$ qua partial metric if and only if it does so in $(X,d)$ qua metric. Put differently, a Cauchy distributor on $(X,d)$ qua $R$-category is neither more nor less than a Cauchy distributor on $(X,d)$ qua $\D(R)$-category of type $0$ (because the type of a Cauchy presheaf $\phi=\lim_{n\to\infty}d(-,x_n)$ on $\bbX=(X,d)$ is necessarily $\lim_{n\to\infty}d(x_n,x_n)=0$); and it is representable qua $R$-enriched distributor if and only if it is qua $\D(R)$-enriched distributor. However, the Cauchy completion of $(X,d)$ qua metric space does not create that ``extra point at infinity'', which the Cauchy completion of $(X,d)$ qua {\it partial} metric space always does!

\subsection{Hausdorff distance, exponentiability}

In \cite{stu10} we developed a general theory of `Hausdorff distance' for quantaloid-enriched categories; applied to the quantaoid $\D(R)$ this produces the following results for partial metrics.
\begin{example}\label{50}
The {\bf Hausdorff space} $\H(X,p)=(\H X,p_{\H})$ of a generalised partial metric space $(X,p)$ is the new generalised partial metric space with elements 
$$\H X=\{S\subseteq X\mid \forall x,x'\in S\:p(x,x)=p(x',x')\}$$ 
(i.e.\ the {\bf typed subsets} of $X$) and partial distance
\begin{equation}\label{50.1}
p_{\H}(T,S)=\bigvee_{t\in T}\bigwedge_{s\in S}p(t,s).
\end{equation}
The inclusion $(X,p)\to\H(X,p)\:x\mapsto\{x\}$ is the unit for the so-called {\bf Hausdorff doctrine} $\H\:\Gen\Par\Met\to\Gen\Par\Met$, and as such enjoys a universal property: it is the universal conical cocompletion (see \cite[Section 5]{stu10}).

The naive extension of the formula in \eqref{50.1} to {\em arbitrary} subsets of $(X,p)$ fails to produce a partial metric, for the following reason. Suppose $a$ and $b$ are elements of a partial metric space $(X,p)$, with $p(a,a)<p(b,b)$. Then $\{a,b\}$ is {\em not} a typed subset of $X$, but if we nevertheless use the sup-inf formula we find in particular that
$$p_\H(\{a\},\{a,b\})=p(a,a),\ p_{\H}(\{a,b\},\{b\})=p(a,b),$$
$$p_\H(\{a\},\{b\})=p(a,b),\ p_\H(\{a,b\},\{a,b\})=p(b,b).$$
In particular is $p_\H(\{a\},\{a,b\})-p_\H(\{a,b\},\{a,b\})+p_\H(\{a,b\},\{b\})\not\geq p_\H(\{a\},\{b\})$, so that $p_\H$ fails to be a partial metric. 
\end{example}
We gave a general characterisation of exponentiable quantaloid-enriched categories and functors in \cite{clehofstu09}; this specialises to the case of partial metric spaces as follows.
\begin{example}\label{51}
A generalised partial metric space $(X,p)$ is {\bf exponentiable} in the (cartesian) category $\Gen\Par\Met$ if and only if 
\begin{equation}\label{51.1}
\begin{array}{l}
\mbox{for all $x_0,x_2\in X$ and $u,v,w\in[0,\infty]$} \\
\mbox{such that $p(x_0,x_0)\vee v\leq u$ and $p(x_2,x_2)\vee v\leq w$:}\\
\bigwedge\Big\{(u\vee p(x_0,x_1))-v+(w\vee p(x_1,x_2))\mid x_1\in X,p(x_1,x_1)=v\Big\}=(u-v+w)\vee p(x_0,x_2).
\end{array}
\end{equation}
This literal application of the very general Theorem 1.1 of \cite{clehofstu09} (but see also Section 5 of that paper) to the specific quantaloid $\D(R)$ can be simplified somewhat. First, using the triangular inequality for the partial metric, it is straightforward to verify that the ``$\geq$'' in \eqref{51.1} always holds. Second, the ``$\leq$'' is trivially satisfied whenever either of $p(x_0,x_2)$, $u$ or $w$ is $\infty$ (because the right hand side is then $\infty$); because $p(x_0,x_0)\vee p(x_2,x_2)\leq p(x_0,x_2)$ we may also exclude the cases where either $p(x_0,x_0)$ or $p(x_2,x_2)$ is $\infty$; and because $v\leq u\wedge w$ (in the hypotheses) we may exclude the case $v=\infty$. The above condition thus becomes:
\begin{equation}\label{51.2}
\begin{array}{l}
\mbox{for all $x_0,x_2\in X$ and $u,v,w\in[0,\infty[$} \\ 
\mbox{such that $p(x_0,x_2)<\infty$, $p(x_0,x_0)\vee v\leq u$ and $p(x_2,x_2)\vee v\leq w$:}\\
\bigwedge\Big\{(u\vee p(x_0,x_1))-v+(w\vee p(x_1,x_2))\mid x_1\in X,p(x_1,x_1)=v\Big\}\leq(u-v+w)\vee p(x_0,x_2).
\end{array}
\end{equation}
It actually suffices to check {\em this} condition only when $p(x_0,x_2)\le u-v+w$. Indeed, whenever $u-v+w<p(x_0,x_2)$ we may apply this (hypothetically valid) condition on $u'-v+w=p(x_0,x_2)$ for the appropriate $u'\geq u$ in the first inequality below, to find that
\begin{eqnarray*}
(u-v+w)\vee p(x_0,x_2)
& = & (u'-v+w)\vee p(x_0,x_2) \\
& \geq & \bigwedge\{(u'\vee p(x_0,x_1))-v+(w\vee p(x_1,x_2))\mid x_1\in X,p(x_1,x_1)=v\}\\
& \geq & \bigwedge\{(u\vee p(x_0,x_1))-v+(w\vee p(x_1,x_2))\mid x_1\in X,p(x_1,x_1)=v\}
\end{eqnarray*}
anyway.
But for $p(x_0,x_2)\le u-v+w$, the inequality in \eqref{51.2} is further equivalent to
$$\bigwedge\{(u\vee p(x_0,x_1))+(w\vee p(x_1,x_2))\mid x_1\in X,p(x_1,x_1)=v\}\leq u+w$$
since $v\le u+w <\infty$ and $\{x_1\in X\mid p(x_1,x_1)=v\}$ cannot be empty. Therefore we finally find that a generalised partial metric space $(X,p)$ is exponentiable in $\Gen\Par\Met$ if and only if
\begin{equation}\label{51.4}
\begin{array}{l}
\mbox{for all $x_0,x_2\in X$, $u,v,w\in[0,\infty[$ and $\varepsilon>0$} \\
\mbox{such that $p(x_0,x_2)\le u-v+w$, $p(x_0,x_0)\vee v\leq u$ and $p(x_2,x_2)\vee v\leq w$} \\
\mbox{there exists $x_1\in X$ such that $p(x_1,x_1)=v$, $p(x_0,x_1)\le u+\varepsilon$ and $p(x_1,x_2)\le w+\varepsilon$.}
\end{array}
\end{equation}
This immediately implies that an exponentiable partial metric space is either empty, or has all distances equal to $\infty$, or has for every $r\in[0,\infty[$ at least one element with self-distance $r$. In particular a generalised metric space $(X,d)$ exponentiable in $\Gen\Par\Met$ if and only if it is empty (even though a non-empty $(X,d)$ may still be exponentiable in $\Gen\Met$!).

Furthermore, with the same proof as in \cite[Theorem 5.3 and Corollary 5.4]{hofrei13}, we obtain that every injective partial metric space (in particular, every partial metric obtained from the presheaf construction in $\Gen\Par\Met=\Cat(\D(R))$, see Subsection \ref{X}) is exponentiable; moreover, the full subcategory of $\Gen\Par\Met$ defined by all injective partial metric spaces is Cartesian closed.
\end{example}

\end{document}